\documentclass[11pt]{amsart}
\usepackage{geometry}                
\usepackage{graphicx}

\usepackage{epstopdf}
\usepackage{amsmath,amssymb,amsthm}

\usepackage{dsfont}
\usepackage{mathrsfs}
\usepackage{hyperref}
\usepackage{tikz}
\usetikzlibrary{matrix,arrows,positioning,calc,shapes}
\usepackage{array}
\usepackage[cmtip,arrow]{xy}
\usepackage{pb-diagram,pb-xy}

\usepackage{algorithm}
\usepackage{algorithmic}
\usepackage{enumerate}

\usepackage{verbatim}

\usepackage{subfigure}

\usepackage[flushleft]{threeparttable}

\theoremstyle{plain}  
\newtheorem{thm}{Theorem}[section]

\newtheorem{prop}[thm]{Proposition}

\theoremstyle{definition}  
\newtheorem{defn}[thm]{Definition}

\newtheorem{ex}[thm]{Example}

\theoremstyle{remark}  
\newtheorem{rem}[thm]{Remark}

\newcommand{\setof}[1]{\left\{ {#1}\right\}}

\newcommand{\bS}{{\bf S}}
\newcommand{\bT}{{\bf T}}

\newcommand{\bRN}{{\bf RN}}

\newcommand{\R}{{\mathbb{R}}}

\newcommand{\cC}{{\mathcal C}}

\newcommand{\cF}{{\mathcal F}}

\newcommand{\cK}{{\mathcal K}}

\newcommand{\cM}{{\mathcal M}}

\newcommand{\cV}{{\mathcal V}}
\newcommand{\cW}{{\mathcal W}}

\newcommand{\cZ}{{\mathcal Z}}

\newcommand{\sP}{{\mathsf P}}

\newcommand{\sMG}{\mathsf{MG}}

\newcommand{\sPG}{\mathsf{PG}}

\newcommand{\FP}{\mathsf{FP}}
\newcommand{\FPon}{\ensuremath{\mathsf{FP}\,\mathsf{ON}}}
\newcommand{\FPoff}{\ensuremath{\mathsf{FP}\,\mathsf{OFF}}}
\newcommand{\FC}{\mathsf{FC}}

\newcommand{\mvmap}{\rightrightarrows}

\newcommand{\supp}[1]{\left| {#1}\right|}

\newcommand{\sMD}{{\mathsf{ MD}}}

\newcommand{\cl}{\mathop{\mathrm{cl}}\nolimits}

\DeclareMathOperator{\sgn}{sgn}

\definecolor{gray85}{gray}{0.85} 
\definecolor{gray8}{gray}{0.8} 
\definecolor{gray7}{gray}{0.7} 
\definecolor{gray6}{gray}{0.6} 
\definecolor{gray5}{gray}{0.5} 
\definecolor{gray4}{gray}{0.4} 
\definecolor{gray35}{gray}{0.35} 

\newcommand{\Sources}{\bS}
\newcommand{\Targets}{\bT}
\newcommand{\CPG}{\mathsf{CPG}}
\newcommand{\GPG}{\mathsf{GPG}}

\title{Combinatorial representation of parameter space for switching networks}
\author{Bree Cummins}
\address{Department of Mathematical Sciences \\
Montana State University\\
Bozeman, MT 59715}
\author{Tomas Gedeon}
\address{Department of Mathematical Sciences \\
Montana State University\\
Bozeman, MT 59715}
\author{Shaun Harker}
\address{Department of Mathematics, Hill Center-Busch Campus\\
Rutgers, The State University of New Jersey \\
Piscataway, NJ  08854-8019, USA}
\author{Konstantin Mischaikow}
\address{Department of Mathematics, Hill Center-Busch Campus\\
Rutgers, The State University of New Jersey \\
Piscataway, NJ  08854-8019, USA}
\author{Kafung Mok}
\address{Department of Mathematics, Hill Center-Busch Campus\\
Rutgers, The State University of New Jersey \\
Piscataway, NJ  08854-8019, USA}

\date{December 13, 2015}                                           

\begin{document}
\maketitle
\begin{abstract}
We describe the theoretical and computational framework for the Dynamic Signatures for Genetic Regulatory Network (DSGRN) database.
The motivation stems from urgent need to understand the global dynamics of biologically relevant signal transduction/gene regulatory networks that have at least 5 to 10 nodes, involve multiple interactions, and decades of parameters.

The input to the database computations is a regulatory network, i.e.\ a directed graph with  edges indicating up or down regulation. 
A computational model based on switching networks is generated from the regulatory network.
The phase space dimension of this model equals the number of modes and the associated parameter space consists of one parameter  for each node (a decay rate), and three parameters for each edge (low level of expression, high level of expression, and threshold at which expression levels change).
Since the nonlinearities of switching systems are piece-wise constant, there is a natural decomposition of phase space into cells from which the dynamics can be described combinatorially in terms of a state transition graph.
This in turn leads to compact representation of the global dynamics called an annotated Morse graph that identifies recurrent and nonrecurrent.
The focus of this paper is on the construction of a natural computable finite  decomposition of  parameter space into domains where the annotated Morse graph description of dynamics is constant. 

We use this decomposition to construct an SQL database that can be effectively searched for  dynamic signatures such as bistability, stable or unstable oscillations, and stable equilibria. 
We include two simple 3-node networks to provide small explicit examples of the type information stored in the DSGRN database.
To demonstrate the computational capabilities of this system we consider a simple network associated with p53 that involves 5-nodes and a 29-dimensional parameter space.
\end{abstract}

%
\section{Introduction}
Though the method presented in this paper is general, our primary motivation arises from the need to understand the global dynamics of signal transduction/gene regulatory networks, e.g.~\cite{kreeger:lauffenburger}.
Our mathematical abstraction of a regulatory network $\bRN$ is a directed graph where the nodes (vertices) $V=\setof{1,\ldots,N}$ indicate the species and the edges $E$ indicate the activation or repression of production of one species by another (this is made precise in Definition~\ref{defn:rn}).
There are at least three fundamental challenges in determining if a given regulatory network provides a biologically relevant model; determining completeness, authentication, and nonlinear interaction.
As discussed in \cite{levine:davidson} completeness and authentication are concerned with whether the relevant species are included in the regulatory network  and whether the proposed interactions are correct. 
Genomic sequencing data  informs completeness, while biochemical knowledge is required for authenticity. 
While the quantity of genomic data is rapidly increasing, detailed biochemical information is still sparse.
Since the interaction between species is typically governed by multiscale processes determining appropriate explicit nonlinearities let alone realistic physical parameters is extremely difficult.

We propose to address these challenges by employing a crude, compact, robust, mathematically rigorous, finitely presented description of dynamics that allows for a combinatorial representation of parameter space.
As a consequence, given a regulatory network $\bRN$, we are able to build the associated {\em Database of Dynamic
Signatures} which codifies the global dynamics over all of parameter space.
The underlying theoretical ideas have already been exploited to study to study a variety of mathematical and biological models \cite{arai2009database,bush2012combinatorial,bush:mischaikow,bush:cowan:harker:mischaikow}.
However, to effectively use these ideas in the context of moderately sized gene regulatory networks where completeness and authentication are in question, we require order of magnitudes greater efficiency for approximating the dynamics and the ability to work with much higher dimensional parameter spaces.
This paper describes and justifies a revised approach that achieves the desired efficiency and provides a natural decomposition of parameter space.

Our starting point has much in common with an approach often referred to as logical networks \cite{albert:collins:glass}.  
The difficulties, alluded to above, in determining and then parameterizing appropriate interactions and nonlinearities has lead to the widespread use of relatively simple models that aim to capture qualitative features of the dynamics.  
The simplest ones are the Boolean models, where each node is represented as either on or off; the dynamics of the $i$-th node consist of evaluation of a logical function defined by this binary representation of the state of the system. 
The evolution of the network proceeds in discrete steps in time which can either synchronously evaluate all functions $\Lambda_i$~\cite{albert2003topology,Chaves2006}, or do so asynchronously~\cite{Calzone2010,Tournier2009,Luo2013}.

Our approach is most closely associated with asynchronous logical networks.
Given a regulatory network $\bRN$ with $N$ nodes,  the associated {\em switching system}~\cite{Glass1972,Glass1973,Glass1978,Gouze2002,Edwards2000,Edwards2012,Edwards2014,DeJong2002} is an $N$-dimensional system of ordinary differential equations  of the form
\begin{equation}
\label{eq:switchingSystem}
\dot{x} = -\Gamma x + \Lambda(x)
\end{equation}
where $\Gamma$ is a diagonal matrix with positive entries, and $\Lambda$ is a piecewise constant function (see Definition~\ref{defn:switching}).  There are a variety of parameters associated with switching networks.
To each variable $x_i$ there is a {\em decay rate} $\gamma_i >0$ (the diagonal terms in $\Gamma$).
To each edge in $\bRN$, corresponding to the impact of species $i$ on $j$, we associate two expression levels, low $l_{j,i}$ and high $u_{j,i}$, and a threshold $\theta_{j,i}$ for $x_i$ at which the expression levels switch.
Thus given a regulatory network  the set of parameters lies in $[0,\infty)^D$ where
$D=N+ 3\cdot \#(E)$, where $N$ is the number of nodes, and $\#(E)$ is the number of edges.
Our goal is to characterize the global dynamics for every point in parameter space.

We hasten to add (this is made clear in the sections that follow) that we do not view \eqref{eq:switchingSystem} as a mathematical model for the biological process; rather \eqref{eq:switchingSystem} is only used to motivate the combinatorial computations that are the focus of this paper.
To emphasize this point we recall that the main result of \cite{gedeon:harker:kokubu:mischaikow:oka} is that applying the methods described in this paper to 2-dimensional switching systems results in a representation of the global dynamics that is valid for a much wider family of nonlinearities, i.e.\ for a system of the form
\begin{equation}
\label{eq:perturbedSystem}
\dot{x} = -\Gamma x + f(x)
\end{equation}
where $f$ is a Lipschitz continuous function. 
Algebraic topological tools such as the Conley index \cite{mischaikow:mrozek} can  be applied to this representation to extract information concerning the structure of invariant sets for dynamics of \eqref{eq:perturbedSystem}.
We also note that for a typical parameter value one can obtain explicit a priori estimates for how much the nonlinearity $f$ can differ from  the switching nonlinearity $\Lambda$.
It is in this sense that we view   \eqref{eq:switchingSystem} as  a computational tool, rather than the mathematical model of biological reality.
An important implication is that we can obtain rigorous results about the dynamics without explicit/detailed knowledge of the appropriate nonlinearity for the biological problem of interest.
It is worth contrasting our approach to more classical methods for relating the dynamics of smoothed systems \eqref{eq:perturbedSystem} with the discontinuous switching systems \eqref{eq:switchingSystem} \cite{edwards2001symbolic, ironi2011dynamics, veflingstad2007analysis}.

To approximate the dynamics of the switching system at a fixed parameter value we use the thresholds $\theta$ to decompose phase space.
These decompositions are then used to determine state transition diagrams (see Section~\ref{sec:transitionDiagram}).  
In this format the dynamics is represented by a sparse directed graph $\cF$ with roughly 
$
\prod_{n=1}^N [O(n) +1]
$
vertices where $O(n)$ is the number of out edges at node $n$.
Given the size of the regulatory networks that we are interested in analyzing, storing all the state transition graphs that arise as one sweeps through parameter space is impractical.
With this in mind we focus on the essential dynamical structures:  the {\em recurrent dynamics}, i.e.\ the nontrivial strongly connected components of $\cF$; and the {\em gradient-like dynamics}, i.e.\ the reachability, defined by paths in $\cF$, between the recurrent components.
There are efficient (both in time and memory) graph algorithms that allow one to identify strongly connected components  (see \cite{arai2009database,bush2012combinatorial} and references therein, and \cite{thieffry} for an application of these techniques in the context of regulatory networks).
Therefore, for the networks we are currently considering this step is not a computational bottleneck.
We encode this information in the form of an annotated {\em Morse graph} $\sMG(\cF)$.
The Morse graph is the minimal directed acyclic graph such that each nontrivial strongly connected component  is represented by a distinct node  and the edges indicate the reachability information inherited from $\cF$ between the nodes (see Section~\ref{subsec:dynamics}).
The annotations consist of optional information, typically problem specific, that allow for easier identification between the dynamics captured by our approach and the dynamics of biological interest.

The computational steps described above are valid for individual parameter values.
A fundamental contribution of this paper is the identification of a natural decomposition of parameter space into regions, called \emph{parameter cells}, over which the transition graphs and hence Morse graphs are constant.
The cells are given in terms of explicit polynomial inequalities in the parameter values and hence take the form of semi-algebraic sets.
Since the global dynamics of switching networks is parameter dependent and we are working with high dimensional parameter spaces it should come as no surprise that understanding the geometry and organizing the structure of all parameter cells is nontrivial.
With this in mind we introduce the {\em parameter graph} (see Section~\ref{sec:parameterGraph}), an undirected graph where each node is identified with a parameter cell and the edges provide information about how the parameter cells are related.
In fact, we make use of two parameter graphs. 
The first, called the \emph{geometric parameter graph} ($\GPG$)   is based on the topology of parameter space as a subset of $[0,\infty)^D$.
This provides a description of the decomposition of parameter space in a language familiar to researchers in the field of dynamical systems. 
The second, called the \emph{combinatorial parameter graph} ($\CPG$) is what is actually computed. 
We prove that there is a graph homomorphism $h\colon \GPG\to \CPG$ and conjecture that the geometric and combinatorial parameter graphs are equivalent, but only have a proof, see Theorem~\ref{thm:isomorphism}, for regulatory networks whose nodes have at most 3 in edges and out edges. 

Our construction of the $\CPG$ for $\bRN$ is based on two facts. First, that the graph structure is actually a canonical graph product over factor graphs $\CPG_n$ which depend only on the local structure of $\bRN$ around each network node $n$. Second, that each $\CPG_n$ is a connected subgraph of a larger graph of combinatorial parameters, and may be constructed via a graph traversal search for \emph{realizable} combinatorial parameters, i.e.\ those that are realized by some geometric parameter of $\bRN$. 
In particular, for each node $n$ in $\bRN$ we compute the set of possible combinatorial parameters, which is determined by the number of out edges at $n$, the number of in edges at $n$, and the logic that determines how the information from the in edges is processed.
To identify whether a particular combinatorial parameter is realizable we make use of cylindrical algebraic decompositions (CAD) \cite{collins1998partial} to determine if there exists a solution to an associated the set of inequalities.
Recall that  a cylindrical algebraic decomposition of a semi-algebraic set is a recursive set of inequalities that defines the elements of the set.
As is discussed in Remark~\ref{rem:CAD}, CAD computations are expensive. 
However, once done they can be re-used at a constant cost. 
Thus our strategy is to perform the CAD computations and store them.
A list of the node structures for which CAD computations have been performed is given in Table~\ref{table:goodnodes}.

Our inability to prove that $h\colon \GPG\to \CPG$ is always an isomorphism, stems from our  lack of general understanding of the geometry of individual and pairs of regions on which transition graphs are constant.
However, the CAD computations, once performed, provide sufficient information to check the conjecture.
This is the essence of the proof of Theorem~\ref{thm:isomorphism}.

The concepts and techniques introduced in this paper have allowed us to develop the $\mathsf{DSGRN}$ (\emph{Dynamic Signatures of Genetic Regulatory Networks})  software \cite{dsgrn} that has the following features and capabilities:
\begin{itemize}
\item It can compute the size of parameter graphs (number of parameter nodes) for any regulatory network constructed using components found in Table~\ref{table:goodnodes}. In particular we supply a web-based program to design such networks which automatically tabulates the size of the parameter graph.
\item It can access a database of cylindrical algebraic decompositions (CAD) of parameter cells corresponding to parameter nodes for regulatory networks constructed using components found in Table~\ref{table:goodnodes}.
\item It can compute annotated Morse graphs given a parameter (cell) and a regulatory network. 
\item It can compute databases of annotated Morse graphs over an entire (combinatorial) parameter graph given a regulatory network. These databases are designed using SQL and support a range of queries over Morse graph attributes and annotations.
\item We supply a web-interface to interact with databases that can filter parameter graphs to show only nodes which satisfy certain queries.
\item We supply a command line interface which allows  access to phase space information for the associated switching system of a regulatory network given a particular parameter of interest.
\item We have supplied documentation of the program along with tutorial materials.
\end{itemize}

To provide the reader with intuition concerning the output in Section~\ref{sec:applications} we consider three regulatory networks.
The first two, the repressilator and the bistable repressilator, consist of 3 nodes and 3 and 4 edges respectively.
For these examples the parameter graphs are sufficiently small that they can be easily visualized.
In general, the output of $\mathsf{DSGRN}$ grows rapidly as a function of the size of the regulatory network and thus can only be accessed efficiently through queries.
To give a sense of the computational capabilities of $\mathsf{DSGRN}$ for problems of biological interest we consider a subnetwork associated with p53 and report the computational times and costs.
The parameter graph information can be accessed at \cite{dsgrn}.

Before concluding this introduction we return to the challenge of determining completeness and authentication where we believe  $\mathsf{DSGRN}$ can be a useful tool.
These challenges imply that in early stages of modeling one cannot  necessarily assume that a proposed regulatory network is `correct.'
By allowing one to compare the output of the model dynamics against experimental data, Boolean models provide a computationally tractable means to attempt to identify the existence of missing species and interaction and/or to exclude non existent interactions and unnecessary species (see for example~\cite{chaves2005,chaves2008,Grieco2013,Niarakis2014}).
Because the computations that we  perform to identify the dynamics is purely combinatorial, the cost of our computations are similar to that of pure Boolean models.
However, because we model using real numbers and ordinary differential equations,  $\mathsf{DSGRN}$  can capture finer dynamical structures that presumably can be more readily identified with experimental values.
More significantly, the fact that $\mathsf{DSGRN}$ provides a complete description of the dynamics over all of parameter space opens up new opportunities for deciding upon the plausibility of or comparison of different  models, e.g.\ how stable is the desired dynamic phenotype to changes in parameters, and for control of the dynamics, e.g.\  which changes in parameters result in a desired dynamic phenotype.
We leave the implementation of these ideas to future works.

%

\section{Switching Networks}
\label{sec:switching}

We review the essential concepts of switching networks. 
In Section~\ref{subsec:networks} we define regulatory networks  and their associated switching systems. 
In Section~\ref{subsec:logic} we discuss the interpretation of the nonlinearities in the switching system as performing logical operations on the inputs. 
In Section~\ref{subsec:cells} we use the discontinuities of the nonlinearity of \eqref{eq:switchingSystem} to impose a natural decomposition of phase space.

\subsection{Regulatory Networks} \label{subsec:networks}
\begin{defn}
\label{defn:rn}
A {\em regulatory network} $\bRN=(V,E)$ is an annotated finite directed graph with vertices $V = \setof{1,\ldots,N}$ called \emph{network nodes} and annotated directed edges $E \subset V\times V \times \{ \to, \dashv \} $ called \emph{interactions}. 
An $\to$ annotated edge is referred to as an {\em activation} and an $\dashv$ annotated edge is called a {\em repression}.
We indicate that either $i \to j$ or $i \dashv j$ without specifying which by writing $(i, j) \in E$. 
An $\to$ annotated edge is referred to as an {\em activation} and an $\dashv$ annotated edge is called a {\em repression}. 
We allow for self edges, but admit at most one edge between any two nodes, e.g.\ we cannot have both $i \to j$ and $i \dashv j$ simultaneously. 
The set of {\em sources} and {\em targets} of a node $n$ are denoted by 
\[
\Sources(n) :=\setof{i \mid (i, n) \in E }\quad\text{and}\quad \Targets(n) := \setof{j \mid (n, j) \in E }.
\] 
The cardinality of $\Sources(n)$ and $\Targets(n)$ are denoted by $\#(\Sources(n))$ and $\#(\Targets(n))$. 
Each node is equipped with a nonlinear function $M_i : \R^{\Sources(i)} \to \R$, called the \emph{logic of  node $i$}. 
\end{defn} 

For convenience we abuse notation and occasionally write a network node as $x_i$ instead of $i$. For instance, we may write $(i, j, \to) \in E$ and $(i, j, \dashv) \in E$ respectively as $x_i \to x_j$ and $x_i \dashv x_j$.

\begin{rem} 
\label{rem:noBlack}
\emph{Throughout} this paper we assume that the regulatory network $\bRN$ does not have any direct negative self-regulation $i \dashv i$ for any $i$. This is done for technical reasons related to the code (see Remark~\ref{rem:noBlack2}).
This is not a serious restriction. 
In biological systems negative self-regulation is often mediated by an intermediary, e.g.\ $x_i \to X_i \dashv x_i$~\cite{edwards2015,cummins2015}. 
Furthermore, future planned developments of the code will allow the user to remove this restriction.
\end{rem}

\begin{defn} \label{def:paramspace} Given a regulatory network $\bRN = (V,E)$, for each edge $(i,j) \in E$ (i.e. $i \rightarrow j$ or $i \dashv j)$ we associate three parameters: $l_{j,i}$, $u_{j,i}$, and $\theta_{j,i}$. (Note the matrix-style subscript order convention.) Additionally, to each  node $i\in V$ we associate a decay rate $\gamma_i$. Each of these parameters are real numbers, so we may regard the collection of all these parameters as a tuple $(l, u, \theta, \gamma) \in \R^{D}$. 
We call this collection of numbers a \emph{parameter} for  $\bRN$.
\end{defn}

\begin{defn}
\label{defn:switching}
Given a regulatory network $\bRN$  \emph{the associated switching system at parameter $(l, u, \theta, \gamma) \in \R^D$}, where $D=N+3\cdot\#(E)$, is given by
\begin{equation}
\label{eq:switching}
\dot{x}_j = -\gamma_j x_j + \Lambda_j(x),\quad j=1,\ldots, N
\end{equation}
where
\begin{equation} \label{eq:lambdaAsComposition}
\Lambda_j := M_j \circ \sigma_j. 
\end{equation}
Here $\sigma_j : \R^N \rightarrow \R^{\Sources(j)}$ is a multi-dimensional step function  defined componentwise (i.e.\ by its coordinate projections $\pi_i(\sigma_j)$) for each $i \in \Sources(j)$ as
\begin{equation}
\label{eq:step_functions}
\sigma_{j,i} = \pi_i(\sigma_{j}(x)) := \begin{cases}
l_{j,i} & \text{if $x_i \to x_j$ and $x_i < \theta_{j,i}$ or if $x_i \dashv x_j$ and $x_i > \theta_{j,i}$ } \\
u_{j,i} & \text{if $x_i \to x_j$ and $x_i > \theta_{j,i}$ or if $x_i \dashv x_j$ and $x_i < \theta_{j,i}$ } \\
\text{undefined} & \text{otherwise.}
\end{cases}
\end{equation}
\end{defn}

\begin{rem}
Switching systems written in the form \eqref{eq:switching} are equivalent to \eqref{eq:switchingSystem} where $\Gamma$ is the diagonal matrix with diagonal entries $\Gamma_{ii} := \gamma_i$ and the $i$-th coordinate of $\Lambda$ is $\Lambda_i$.
\end{rem}

The dependence on $x$ of the right-hand-side of \eqref{eq:switching} involves the expression $\sigma_j(x)$. Since $\sigma_j$ is a multidimensional step function which compares variables to thresholds, a grid-like structure is imposed upon phase space.

\begin{defn} 
Let $z = (u,l,\theta,\gamma)$ be a parameter for the regulatory network \bRN. For each $i \in V$, observe the convention that $\theta_{-\infty,i} = 0$ and $\theta_{\infty,i} = \infty$. For all $i\in V$, $j_1$, $j_2 \in V \cup \{-\infty, \infty\}$, we say $\theta_{j_1, i}$ and $\theta_{j_2, i}$ are \emph{consecutive thresholds} if $\theta_{j_1, i} < \theta_{j_2, i}$ and there does not exist $\theta_{j, i}$ such that $\theta_{j_1, i} < \theta_{j, i} < \theta_{j_2, i}$. 
For each $i = 1, 2, \cdots, N$, suppose $\theta_{a_i, i}$ and $\theta_{b_i, i}$ are consecutive thresholds. Then we say that the product of open intervals
\[ \kappa := \prod_{i=1}^N (\theta_{a_i, i}, \theta_{b_i, i}) \] is a \emph{fundamental cell} of the regulatory network. We denote the collection of fundamental cells as $\cK(z)$. If $a_j \in V$ then we say that the bounded hyperplane
\[ \kappa^-_j := \prod_{i=1}^{j-1} (\theta_{a_i, i},\theta_{b_i, i}) \times \{\theta_{a_j, j}\} \times \prod_{i=j+1}^{N} (\theta_{a_i, i},\theta_{b_i, i}) \] is a \emph{left face} of $\kappa$ with projection index $j$ and switching index $a_j$. Similarly, if $b_j \in V$ then we say the bounded hyperplane \[ \kappa^+_j := \prod_{i=1}^{j-1} (\theta_{a_i, i},\theta_{b_i, i}) \times \{\theta_{b_j, j}\} \times \prod_{i=j+1}^{N} (\theta_{a_i, i},\theta_{b_i, i}) \] is a \emph{right face} of $\kappa$ with projection index $j$ and switching index $b_j$. A \emph{face} of a fundamental cell $\kappa$ is either a left or right face of $\kappa$. 
\end{defn}

We restrict our focus to non-negative parameters $\bar{Z} \subset [0,\infty)^{D}\subset \R^{D}$   for which we interpret the $l$ and $u$ values as the \emph{lower} and \emph{upper} values that may be taken, respectively. 
This  imposes the additional requirement that $l_{j,i} \leq u_{j,i}$. 

\begin{defn}\label{def:regular_param}
Given a switching network $\bRN$ the associated \emph{parameter space $\bar{Z}$} is defined to be the collection of all parameters $(l,u,\theta,\gamma) \in [0,\infty)^{D}$ for which $l_{j,i} \leq u_{j,i}$ for all $(i,j)\in E$. 
A parameter $z = (l,u,\theta,\gamma) \in \bar{Z}$ is  \emph{regular} if 
\begin{enumerate}
\item the inequalities are satisfied strictly, i.e.\ $0 < l_{j,i} < u_{j,i}$, $0 < \gamma_i$, and $0 < \theta_{j,i}$,
\item for each fixed $i$ the threshold values $\theta_{j,i}$ are distinct, and
\item for each $\kappa \in\cK(z)$ the value  $ \Lambda_i (\kappa) \not = \gamma_i \theta_{j,i}$ for each threshold $\theta_{j,i}$ that defines $\kappa $.
\end{enumerate}
We denote the collection of regular parameters by $Z$. 
Notice that  $\bar{Z}$ is (as the notation suggests) the topological closure of $Z$ in $\R^{D}$ and the set $Z$ is generic in $\bar{Z}$. 
For a regular parameter $z\in Z$ the thresholds $\{ \theta_{j,i} : j \in \Targets(i)\}$ occur in some definite (total) order for each $i \in V$. 
We denote this ordering by $O(z)$ and the collection of all  orderings over all of parameter space as 
\[
O(Z) := \bigcup_{z\in Z} O(z). 
\]
\end{defn}

\subsection{Network Node Logics} \label{subsec:logic}

Definition \ref{defn:switching} does not specify the nonlinear functions $M_j$. 
As indicated in the introduction our approach is associated with the interpretation of regulatory networks as logical networks. 
To be more precise, a logical expression involving truth variables $v_i$, logical conjunctives $\wedge$ (i.e. ANDs), and logical disjunctives $\vee$ (i.e. ORs) leads to an analogous arithmetic expression by replacing $\wedge$ with $\cdot$ and $\vee$ with $+$. 
For example, $(a \vee b) \wedge c$ becomes $(a + b)c$. 
Observe that given truth variables $v_i$, a logical expression $\ell(v_1, v_2, \cdots, v_n)$ (without negations) leads unambiguously to the multilinear arithmetic expression $M(x_1,x_2,\cdots,x_n)$ given by
\[ 
M(x_1,x_2,\cdots,x_n) := \sum_{ \ell(v_1, \cdots, v_n) = T } \prod_{v_i = F} (1-x_i) \prod_{v_i = T} x_i, 
\]
where $x_i\in\R$.

Note that for every logical expression where each variable occurs \emph{at most once} the recipe of replacing $\wedge$ with $\cdot$ and $\vee$ with $+$  produces an arithmetic expression which is equivalent to this multilinear expression.

For the purposes of this paper our focus is on regulatory networks where one considers a \emph{logic} for each network node $j$ consisting of a logical expression involving each of the variables $x_i \in \Sources(j)$ precisely once. 
The multilinear functions $M_j$ appearing in \eqref{eq:lambdaAsComposition} are obtained from these logical expressions.

\begin{rem}
\label{rem:monomial}
These assumptions on the structure of the terms in $M_j$ imply that $\Lambda_j$ can always be expressed in the form of a sum of monomials involving the step functions $\sigma_{j,i}(x)$ where the degree of $\sigma_{j,i}(x)$ is either zero or one.
\end{rem}

\subsection{Fundamental Cells and Vector Fields} \label{subsec:cells}

If we restrict to a fundamental cell then \eqref{eq:switching} reduces to a much simpler linear form. In particular, observe that given a fundamental cell $\kappa$, if $x,x'\in \kappa$ then $\Lambda(x) = \Lambda(x')$, and therefore $\Lambda(\kappa)$ is well-defined. In accordance with this observation we make the following definitions:

\begin{defn}
A regulatory network $\bRN$ and a choice of parameter values $z\in Z$ leads to a uniquely defined  switching system \eqref{eq:switchingSystem} and set of fundamental cells $\cK(z)$.
The \emph{$\kappa$-cell vector field} for a fundamental cell $\kappa\in \cK(z)$ is given by
\begin{equation}\label{eq:kappa}
f^{\kappa}(x) := -\Gamma x + \Lambda(\kappa).
\end{equation} 
We denote the flow generated by \eqref{eq:kappa} by $\psi_\kappa$.
Observe that
\[
\Phi(\kappa):= \Gamma^{-1}\Lambda(\kappa)
\]
is a global attracting fixed point for \eqref{eq:kappa}. 
Accordingly we say that a fundamental cell $\kappa$ is an \emph{attracting cell} if
\[
\Gamma^{-1}\Lambda(\kappa) \in \kappa.
\]
\end{defn}

%
\section{State Transition Diagrams}
\label{sec:transitionDiagram}

As indicated in the introduction we capture the dynamics of \eqref{eq:switching} via  \emph{state transition diagrams}, directed graphs where the vertices correspond to regions of phase space, and the edges indicate how regions are related by the dynamics.
We begin in Section \ref{subsec:wall} by defining wall-labelings that encapsulate combinatorial information derived from the $\kappa$-cell vector fields \eqref{eq:kappa}.
In Section \ref{subsec:diagrams} we give three different constructions of state transition diagrams. 
In Section \ref{subsec:dynamics} we show that these three constructions are equivalent in the sense that they lead to equivalent dynamical information.

\subsection{Wall labelings} \label{subsec:wall}

Faces of fundamental cells play a key role in our combinatorial representation of the dynamical system \eqref{eq:switching}. However each such face has two adjacent fundamental cells and the $\kappa$-cell vector fields on either side may differ. Accordingly we refine our concept of face to make reference to one of the adjacent fundamental cells.

\begin{defn} 
A \emph{wall} is a pair $(\tau, \kappa)$ where $\kappa$ is a fundamental cell and $\tau$ is a face of $\kappa$.
Each wall inherits the projection and switching indexes from the corresponding face $\tau$ of $\kappa$. We say the \emph{sign of the wall $(\tau, \kappa)$ is $1$} (and write $\sgn (\tau,\kappa) = 1$) if $\tau$ is a left face of $\kappa$ and we say the sign of the wall is $-1$ if $\tau$ is a right face of $\kappa$ (and write $\sgn (\tau,\kappa) = -1$). For a fixed parameter value $z$ we denote the collection of walls by $\cW(z)$. 
\end{defn}

Observe that given a wall $(\tau, \kappa)$  there are three possibilities with respect to the  $\kappa$-cell vector field $f^\kappa$: it is everywhere tangential to $\tau$;  it points out of $\kappa$ everywhere on $\tau$; or it points into $\kappa$ everywhere on $\tau$. 
If the projection index of $\tau$ is $i$, then these three cases are determined by the sign of the expression $f^{\kappa}_i(\tau)$ and whether the wall corresponds to a left or right face (i.e. the sign of the wall) which in turn can be determined as a function of parameters. We summarize this in the following definition.

\begin{defn} 
Consider a switching network at a parameter value $z\in Z$. The  \emph{wall-labeling} of $\cW(z)$ is the function $\ell : \cW(z) \rightarrow \setof{-1, 0, 1}$ defined as follows.
Let $(\tau,\kappa)\in\cW(z)$ have projection index $i$ and switching index $j$. Then define
\begin{equation} 
\label{eqn:sign}
 \ell((\tau, \kappa)) := \sgn (\tau, \kappa) \cdot \sgn(f^\kappa_i(\tau)) = \sgn (\tau, \kappa) \cdot \sgn\left( - \gamma_i \theta_{j, i} + \Lambda_i(\kappa)\right).
\end{equation}
Note the last equality follows since if $x \in \tau$ then $x_i = \theta_{j,i}$.
\end{defn}

\begin{rem}
\label{rem:wall}
As \eqref{eqn:sign} makes clear the wall-labeling function depends explicitly on the choice of parameters for the switching network. 
However, given any two parameter values for which the ordering of the thresholds is the same there is an obvious identification of the fundamental cells and walls.
Using this identification, the wall labelling is completely determined by the values of $\sgn(f^\kappa_i(\tau))$ over the collection of fundamental cells.
As such  we can define equivalence classes of parameter values over which 
wall-labelings are constant.
\end{rem}

\begin{defn}
A wall $(\tau, \kappa)$ is an \emph{absorbing wall} if $\ell(\tau, \kappa) =-1$, an \emph{entrance wall} if $\ell(\tau, \kappa) = 1$, and a \emph{tangential wall} if $\ell(\tau, \kappa) =0$.
\end{defn}

For parameters in the set $Z$ (i.e. regular parameters) we will have only absorbing and entrance walls.

\begin{prop} 
Given a switching system with parameter $z\in Z$, there are no tangential walls. That is, the wall-labeling function $\ell$ satisfies
\[
\ell(\tau,\kappa) \neq 0
\]
for all $(\tau,\kappa)\in\cW(z)$.
\end{prop}
 
The classification of walls according to the value of the wall-labeling function arises from geometric considerations of the flows on the fundamental cells. We leave the proof of the following to the reader: 

\begin{prop} 
\label{prop:att=l}
A fundamental cell $\kappa$ is attracting if and only if every wall $(\tau,\kappa)\in\cW$ is an entrance wall. 
\end{prop}

A stronger relation between the labeling of walls and the dynamics of the $\kappa$-equation is as follows. Again the proof is left to the reader.

\begin{prop}
Let $x\in\kappa\in\cK(z)$ where $z\in Z$ is a regular parameter value. Recall that $\psi_\kappa$ is the flow generated by~\eqref{eq:kappa}.
If $\kappa$ is an attracting cell, then there exists a unique time $t^-_x<0$ such that $\psi_\kappa((t^-_x,\infty),x)\subset \kappa$ and $\psi_\kappa(t^-_x,x)\in \bar{\tau}$ where $(\tau,\kappa)$ is an entrance wall and $\bar{\tau}$ denotes closure of $\tau$.
If $\kappa$ is not an attracting cell, then there exist unique times $t^-_x<0<t^+_x$ such that $\psi_\kappa((t^-_x,t^+_x),x)\subset \kappa$,  $\psi_\kappa(t^-_x,x)\in \bar{\tau}$ where $(\tau,\kappa)$ is an entrance wall, and $\psi_\kappa(t^+_x,x)\in \bar{\tau}'$ where $(\tau',\kappa)$ is an absorbing wall.
\end{prop}

\subsection{State Transition Diagram Constructions} \label{subsec:diagrams}

Recall that a state transition diagram is a directed graph. 
To emphasize that we employ this as a means of representing information about dynamics we adopt an equivalent perspective: a state transition diagram is a combinatorial multivalued map $\cF\colon \cV \mvmap \cV$ (where $\cV$ is the collection of vertices) such that $\nu'\in\cF(\nu)$ if and only if there is a directed edge $\nu \rightarrow \nu'$ in the state transition diagram. Using the multivalued map notation the existence of a path from $\nu$ to $\nu'$ is expressed by $\nu'\in \cF^k(\nu)$ for some positive integer $k$. 

Let $\ell$ be the wall-labeling for a switching network at a fixed parameter.
We give three constructions for state transition diagrams. 
In each case, note that the state transition diagrams depends only on $\ell$, and hence it is through the wall-labelling that the state transition diagrams inherit their dependence on parameters. 
In fact, as indicated by Remark~\ref{rem:wall}, this inheritance remains constant on equivalence classes of parameter values.

\begin{defn}\label{defn:wallgraph}
The \emph{wall graph $\cF\colon \cV \mvmap \cV$ induced by the wall-labeling $\ell$} is defined as follows.
There is a bijection that identifies the set of vertices $\cV$ with the collection of attracting fundamental cells and the faces of all fundamental cells.
For each pair of faces $\tau$, $\tau'$ admitting a fundamental cell $\kappa$ such that $(\tau, \kappa)$ is an entrance wall and $(\tau', \kappa)$ is an absorbing wall, there is an edge $\tau \rightarrow \tau'$, or equivalently $\tau'\in\cF(\tau)$.
For each attracting domain $\kappa$, $\kappa \in \cF(\kappa)$ and $\kappa\in\cF(\tau)$ for each wall $(\tau,\kappa)$.
 \end{defn}
 
\begin{rem}
\label{rem:noBlack2}
A consequence of the assumption that the regulatory network does not have any direct negative self-regulation  (see Remark~\ref{rem:noBlack}) is that every node in the wall graph has an an out edge.
\end{rem}

\begin{defn} 
The \emph{domain graph $\cF\colon \cV \mvmap \cV$ induced by the wall-labeling $\ell$} is defined as follows. 
There is a bijection that identifies the set of  vertices $\cV$ with the collection of fundamental domains $\cK(\kappa)$. 
If some fundamental domain $\kappa$ is an attracting domain, then $\kappa \in \cF(\kappa)$. 
Furthermore, $\kappa'\in \cF(\kappa)$ whenever there exists a face $\tau$ such that $(\tau, \kappa)$ and $(\tau, \kappa')$ are walls such that $\ell((\tau, \kappa)) = -1$ (indicating an absorbing wall of $\kappa$) and $\ell((\tau, \kappa')) = 1$ (indicating an entrance wall of $\kappa'$). 
\end{defn}

\begin{defn} 
The  \emph{wall-domain graph $\cF\colon \cV \mvmap \cV$ induced by the wall-labeling $\ell$} is defined as follows.
There is a bijection that identifies the set of  vertices $\cV$ with the collection of attracting fundamental cells and the faces of all fundamental cells.
There are three types of edges.
If $(\tau, \kappa)$ is absorbing wall, then $\tau\in\cF(\kappa)$.
If $(\tau,\kappa)$ is an entrance wall, then $\kappa\in\cF(\tau)$.
Finally, if $\kappa$ is an attracting fundamental domain, then $\kappa\in \cF(\kappa)$.
\end{defn}

\subsection{Dynamical Signatures} \label{subsec:dynamics}

As indicated in the introduction, to store the dynamics we make use of a more compact representation.

\begin{defn}\label{def:recurrent} 
A \emph{recurrent component} (also referred to as a \emph{strongly connected path component}) of a directed graph $\cF$ is a maximal collection $\cC$ of vertices such that for any $u, v \in \cC$ there exists a non-empty path from $u$ to $v$ in $\cC$. In the context of dynamical systems we refer to a recurrent component of $\cF$ as a {\em Morse set} of $\cF$ and denote it by $\cM\subset \cV$. The collection of all recurrent components of $\cF$ is denoted by 
\[
\sMD(\cF) :=\setof{\cM(p)\subset \cV\mid p\in \sP}
\]
and is called a {\em Morse decomposition} of $\cF$. Here $\sP$ is an index set. Recurrent components inherit a well-defined partial order by the reachability relation in the directed graph $\cF$. Specifically, we may write the partial order on the indexing set $\sP$ of $\sMD(\cF)$  by defining
\[
q \leq p\quad \text{if there exists a path in $\cF$ from an element of $\cM(p)$ to an element of $\cM(q)$}.
\]
\end{defn}

Primarily for clarity we note the following facts regarding recurrent components:

\begin{prop} Two elements $\nu,\nu'\in\cV$ belong to the same recurrent component of $\cF$ if and only if there exist positive integers $k,k'$ such that $\nu'\in \cF^k(\nu)$ and $\nu\in \cF^{k'}(\nu')$. Distinct recurrent components are disjoint. Not every vertex need belong to some recurrent component. Recurrent components are strongly connected components. The only strongly connected components that are not recurrent components are singleton sets that do not have a self-edge.
\end{prop}

\begin{defn} \label{defn:morsegraph}    
The \emph{Morse graph} of $\cF$, denoted $\sMG(\cF)$, is the Hasse diagram of the poset $(\sP,\leq)$. We refer to the elements of $\sP$ as the \emph{Morse nodes} of the graph.\end{defn}

We note that the computation of Morse graphs is feasible and can be accomplished via well-known algorithms for strongly connected components \cite{tarjan1972depth} and transitive reduction \cite{aho1972transitive}.

\begin{prop} The Morse graphs induced by the wall graph, the wall-domain graph, and the domain graph are isomorphic.
\end{prop}
\begin{proof} In the wall graph, every entrance face maps to every absorbing face. Meanwhile in the wall-domain graph every entrance face maps to the fundamental cell which then maps to each of its exit faces. It follows that the reachability between faces in the wall graph and the wall-domain graph is the same. What is more, every recurrent component in the wall graph either contains a face or else is an attracting domain. From this one may establish the isomorphism between the Morse graph of a wall graph and the Morse graph of the wall-domain graph. 
Establishing the isomorphism of Morse graphs induced by the wall-domain graph and the domain graph is similar: we observe that the reachability between fundamental domains is the same in either and that each recurrent component must contain at least one domain. Combining these two isomorphisms gives the isomorphism from wall graph to domain graph.
\end{proof}

In light of the previous result and the simplicity of the domain graph one might wonder why we should bother with the wall graph or wall-domain graph at all. The answer is two-fold. First, it is possible to refine our analysis so that in the wall graph not every entrance face $(\tau,\kappa)$ is mapped to every absorbing face $(\tau',\kappa)$ for a fundamental cell $\kappa$. For example, consider Figure \ref{fig:translation}. In (a), we see the wall graph obtained in the present work in the situation of a 2D fundamental cell with two entrance and two exit faces arranged as indicated. However, for a particular parameter the actual trajectories would correspond to the diagrams sketched in (b) and (c). To achieve combinatorial descriptions capturing the additional information in (b) or (c) the domain graph is inadequate and it will be necessary to use notions such as the wall graph. We leave this to future work.

\begin{figure}
\begin{tikzpicture}
[scale=2.0]

	\draw[-] (0.0,0.0) -- (1.0,0.0);
	\draw[-] (0.0,0.0) -- (0.0,1.0);
	\draw[-] (1.0,0.0) -- (1.0,1.0);
	\draw[-] (0.0,1.0) -- (1.0,1.0);

	\draw[->] (0.0,0.5) -> (0.45,0.95);
	\draw[->] (0.0,0.5) -> (0.95,0.5);
	\draw[->] (0.5,0.0) -> (0.5,0.95);
	\draw[->] (0.5,0.0) -> (0.95,0.45);
\draw(0.5,-0.5) node{(a)};

	\draw[-] (2.0,0.0) -- (3.0,0.0);
	\draw[-] (2.0,0.0) -- (2.0,1.0);
	\draw[-] (3.0,0.0) -- (3.0,1.0);
	\draw[-] (2.0,1.0) -- (3.0,1.0);

	\draw[->] (2.0,0.5) -> (2.45,0.95);
	\draw[->] (2.0,0.5) -> (2.95,0.5);
	\draw[->] (2.5,0.0) -> (2.95,0.45);
\draw(2.5,-0.5) node{(b)};

	\draw[-] (4.0,0.0) -- (5.0,0.0);
	\draw[-] (4.0,0.0) -- (4.0,1.0);
	\draw[-] (5.0,0.0) -- (5.0,1.0);
	\draw[-] (4.0,1.0) -- (5.0,1.0);

	\draw[->] (4.0,0.5) -> (4.45,0.95);
	\draw[->] (4.5,0.0) -> (4.5,0.95);
	\draw[->] (4.5,0.0) -> (4.95,0.45);
\draw(4.5,-0.5) node{(c)};

\end{tikzpicture}
\caption{(a) Coarse wall graph where all entrance faces map to all exit faces. (b) Finer wall graph with a disallowed entrance-exit trajectory from bottom to top (c) Finer wall graph with disallowed trajectory from left to right}
\label{fig:translation}
\end{figure}
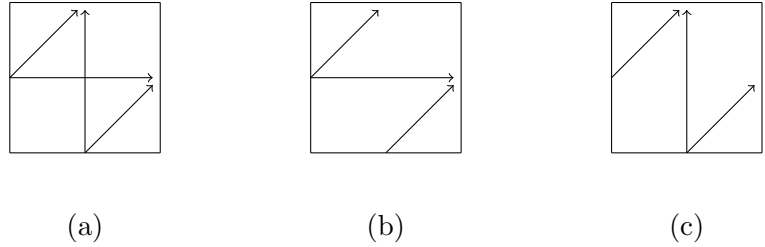

Second, via more sophisticated constructions of the directed graph we can obtain information about unstable dynamics from the Morse graph. These constructions involve using lower-dimensional cells such as faces as vertices and will be be detailed in a future work.

\subsection{Annotation of Morse Nodes}
\label{sec:annontation}

In addition to recording the Morse graph for a parameter $z \in Z$ it is possible to produce extra information in the form of annotations we associate with the Morse nodes of the Morse graph. 
We refer to this information, Morse graphs plus annotations, as a \emph{dynamical signature}. 
Presently, we compute annotations for each Morse node based on the vertices present in the associated Morse set. We briefly describe these annotations. First, we say that  a Morse set \emph{makes an $x_d$ transition} if it contains vertices corresponding to cells with differing $x_d$ coordinates. We make annotations according to the set of transitions. In the simplest case, there are no transitions (we have only a single attracting cell) and we annotate the Morse set as a fixed point but respecting three subcases: (a) if the fundamental cell is located to the left of each threshold in all dimensions the Morse node is marked $\mathsf{FP\,OFF}$ (\emph{fixed point off}); (b) if the fundamental cell is located to the right of at least one threshold in each dimension it is marked $\mathsf{FP\,ON}$ (\emph{fixed point on}); (c) otherwise it is marked just $\mathsf{FP}$. In the other extreme case every transition is made (i.e. $x_1, x_2, \cdots, x_N$). In this case we annotate the Morse node $\mathsf{FC}$ (for \emph{full cycle}). Otherwise we annotate the Morse node according to the subset of variables for which there is a transition.

%

\section{Parameter Graph}
\label{sec:parameterGraph}

The content of Sections~\ref{sec:switching}  and \ref{sec:transitionDiagram} implies that given a regulatory network $\bRN$ and a  parameter $z \in Z$ for the associated switching system \eqref{eq:switching} it is possible to create an annotated Morse graph.

Since two parameters $z, z' \in Z$ give rise to the same annotated Morse graph provided they give rise to the same wall-labeling (up to the equivalence indicated in Remark~\ref{rem:wall}) it is natural to discretize parameter space according to regions that are guaranteed to give the same wall-labeling. 
How to accomplish this is the topic of this section.
As indicated in the introduction, the final results are parameter graphs that are defined below.

The nodes of a parameter graph, which we denote by $\cZ$, are meant to correspond to regions in parameter space and the edges are meant to indicate  geometric relations between the regions.
The resulting \emph{database of dynamic signatures} can be viewed as a map
 \[ 
 \mathsf{DB} : \cZ \rightarrow \mathsf{AnnMG},
  \] 
where  $\mathsf{AnnMG}$ denotes the collection of annotated Morse graphs. 
As indicated in Section~\ref{sec:applications}, in applications queries to the database are often concerned with finding the set of nodes which correspond to a particular annotated Morse graph.

We remark that the parameter space $Z\subset [0,\infty)^D$, as defined in Definition~\ref{def:regular_param}, is a semi-algebraic set, i.e.\ it is expressed in terms of polynomial inequalities. 
To see this note that by Definition \ref{def:paramspace}, $\bar{Z}$ is a semi-algebraic set and $Z$ is the complement in $\bar{Z}$ of the set of parameters that satisfy any one the equalities
\begin{align}
0 & = l_{j,i} - u_{j,i}, \nonumber \\
0 & = \gamma_i, 0 = \theta_{j,i}, 0 = l_{j,i}, 0 = u_{j,i}, \nonumber \\
0 & = \theta_{j,i}-\theta_{j',i}\quad \text{for distinct $j,j'\in\Targets(i)$, or} \label{eq:equalities} \\
0 & = \gamma_i\theta_{j,i} - \Lambda_i(\kappa) \quad\text{where $\theta_{j,i}$ defines a face of $\kappa$;} \nonumber
\end{align}
where $\kappa\in \cK(z)$ and by Remark~\ref{rem:monomial} $\Lambda_i(\kappa)$ can be expressed as a sum of monomials involving parameters.

We propose two conceptually different means of identifying $Z$ with $\cZ$.
In Section~\ref{subsec:geometricparametergraph} we describe an approach based on topological considerations;  the basic elements are connected components of $Z$ and their adjacency structure is defined in terms of closures of these sets. 
This gives rise to the \emph{Geometric Parameter Graph} ($\GPG$).

In Section~\ref{subsec:combinatorialparametergraph} we describe an approach that is explicitly computable; we consider subsets of $Z$ described by systems of strict inequalities. We define an adjacency structure by considering reversing the direction of these inequalities one at a time. This gives rise to  the \emph{Combinatorial Parameter Graph} ($\CPG$).

In Section~\ref{subsec:factorization} we present a decomposition of the geometric and combinatorial parameter graphs as a product of smaller graphs in a way which corresponds to the structure of the regulatory network. 
In Section~\ref{subsec:equivalence} we compare the $\GPG$ and $\CPG$ graphs. 
There is a canonical graph homomorphism $h:\GPG \to \CPG$ induced from inclusion, but it is an open question whether or not this homomorphism is in fact always an isomorphism (i.e. we have not proven in general that these two approaches are equivalent). 
We are, however, able to give a class of networks for which $\GPG$ and $\CPG$ are known to be the same. 
Finally, in Section~\ref{subsec:computation} we will discuss the computation of combinatorial parameter graphs.

\subsection{Geometric Parameter Graph} \label{subsec:geometricparametergraph}

The state transition diagram $\cF(\cdot,z)\colon \cV(z)\mvmap \cV(z)$ is defined for each $z\in Z$ and is completely determined by the wall-labeling \eqref{eqn:sign}. The ordering $O$ and the wall-labeling is constant on connected components of $Z$. 
This motivates a combinatorialization of parameter space via the connected components, which we label by $\cZ$.
With this in mind  we define the \emph{parameter nodes} of  \emph{Geometric Parameter Graph} ($\GPG$) to be $\cZ$.  
To complete the definition of the $\GPG$ we require a notion of adjacency for parameter nodes. 

\begin{defn} \label{defn:parameterGraph} We say that a parameter value $z \in \bar{Z}$ is \emph{$k$-deficient} if exactly $k$ of the equalities of \eqref{eq:equalities} are satisfied.
Given a switching network \eqref{eq:switchingSystem} the associated \emph{geometric parameter graph} $\GPG$ has vertices $\cZ$ and edges $(\zeta,\zeta')$ if there exists $z\in \cl(\zeta)\cap \cl(\zeta')$ such that $z$ is $1$-deficient. Here the closures $\cl(\zeta)$ and $\cl(\zeta')$ are taken in $\bar{Z}$.
\end{defn}

\begin{rem} We believe that the concept of $k$-deficient parameter values can be used generate a regular CW-decomposition of $\bar{Z}$ where the elements of $\cZ$ represent the cells of  dimension $D$. 
In this context our notion of adjacency in the $\GPG$ corresponds to cells that share a $D-1$ dimensional face. 
Extending this construction would provide a means of understanding the topology of regions of parameter space that are associated with specified dynamic phenotypes.
\end{rem}

The following example illustrates the concept of the geometric parameter graph.

\begin{ex}
\label{ex:hyparrange}
Consider the simplest regulatory network $\bRN =(X,E)$ where $X=\setof{1}$, $1 \to 1$, and $M_1(x) = x$.
The regulatory network is shown in Figure~\ref{fig:RN1}(a).
Since the $1\to 1$ edge is activating, the associated switching network takes the form
\begin{equation}
\label{eq:RN1}
\dot{x} = -\gamma x + \begin{cases} l & \text{if $x<\theta$,} \\ u & \text{if $x>\theta$.}\end{cases}
\end{equation}
The associated phase space and subdivision is shown in Figure~\ref{fig:RN1}(b).
In particular, there are two fundamental cells $\kappa_1 := (0,\theta)$ and $\kappa_2 := (\theta,\infty)$.
The dimension of the associated parameter space is $D=1+3 = 4$ and 
\[ 
Z=\setof{(l, u,\theta,\gamma)\mid l<u, \gamma\theta \notin \{l, u\} }\subset [0,\infty)^4. 
\]

Because of the simplicity of the problem, it is easy to enumerate all the state transition diagrams and compute the associated Morse graphs.
Figure~\ref{fig:RN1}(c) indicates the annotated Morse graphs.
Without the annotations Morse graphs $\sMG(1)$ and $\sMG(3)$ are equivalent and consist of a single node.
However, in the process of the computation we can identify that the fixed point is in $\kappa_1$ is less than the threshold, thus the node $x$ is in an `off' state. 
We denote this information by $\mathsf{FP\,OFF}$.
The annotated Morse graph includes this information at the node.
Similarly, the annotated Morse graph $\sMG(3)$ indicates that the fixed point is in $\kappa_2$.
We denote this information by $\mathsf{FP\,ON}$.
The Morse graph $\sMG(2)$ has two minimal nodes generated by two attracting cells in one of which the fixed point is less than the threshold and in the other the fixed point is greater than the threshold. Again these fixed points are annotated $\mathsf{FP\,OFF}$ and $\mathsf{FP\,ON}$, respectively.

The associated regions of parameter space, i.e.\ their defining inequalities, are indicated in Figure~\ref{fig:RN1}(d).
Observe that there is a straight line in $PN(i)$ from any point in $PN(i)$  to the origin.
Thus one can show that each region  $PN(i)$ is connected.
Thus, the $\GPG$ contains 3 nodes.
Observe that if $z\in\bar{Z}$ satisfies $\gamma\theta - u =0$, then $z\in \cl(PN(1))\cap \cl(PN(2))$ and 
similarly, if $z\in\bar{Z}$ satisfies $\gamma\theta - l =0$, then $z\in \cl(PN(2))\cap \cl(PN(3))$.
Thus, we have an edge between the nodes corresponding to $PN(1)$ and $PN(2)$ and an edge between the nodes corresponding to $PN(2)$ and $PN(3)$.
Therefore, Figure~\ref{fig:RN1}(d) is the $\GPG$ for this example.

\end{ex}

How to extend these arguments in Example~\ref{ex:hyparrange} to more complicated regulatory networks is not obvious.
However, consider $Z' := \setof{(l,u,\theta) \in [0,\infty)^3 \mid (l,u,\theta,1)\in Z} \subset [0,\infty)^3$ and observe that characterizing $Z'$ is equivalent to characterizing the complement of the degenerate finite hyperplane arrangement \cite{stanley} 
\[
\setof{ \setof{ v\in[0,\infty)^3 \mid a\cdot v =0 } \mid a\in A}
\]
where 
\[
A = \setof{(1,0,0), (0,1,0), (0,0,1), (1,-1,0),(1,0,-1),(0,1,-1)}.
\]
In this setting it is fairly easy to determine that $Z'$ consists of three unbounded connected components for which the origin is contained in their closures.
However, we are interested in characterizing all of $Z$, which implies that we need to consider complements of the nonlinear equations $-\gamma\theta + l =0$ and $-\gamma\theta + u=0$.
For a general regulatory network the dimension $D$ grows linearly in the number of vertices and edges and the terms of $\Lambda$, while multilinear, may consist of higher dimensional products of the parameters. 
This implies that the problem of understanding $Z$ is at least as challenging as that of determining the cells in a degenerate finite hyperplane arrangement. 
To deal with these complications, in the next subsections we turn to techniques from computational algebraic geometry. 
For now we content ourselves with the following result.

\begin{figure}
\begin{tikzpicture}
[scale=0.4]
	\draw (0,3.5) circle(1 cm);
	\draw (0,3.5) node{$x_1$};
			
	\draw[<-] (0.75,4.5) arc (-45:225:1 cm);
	
	\draw(0,0) node{(a)};
\end{tikzpicture}
\qquad
\begin{tikzpicture}
[scale=0.4]
	\draw[thick](0,3.5) -- (8,3.5);
	\draw(9,3.5) node{$x$};

	\draw (4,3) -- (4,4);
	\draw (0,3) -- (0,4);

	\draw (0,2.5) node{$0$};	
	\draw (4,2.5) node{$\theta$};
	\draw (2,2.5) node{$\kappa_1$};
	\draw (6,2.5) node{$\kappa_2$};

	\draw(4,0) node{(b)};
\end{tikzpicture}
\\
\
\\
\begin{tikzpicture}
[scale=0.4]
	\draw[thick] (-1,1) arc (0:360:2 and 1);
	\draw (-3,1) node{$\mathsf{FP\,OFF}$};
	\draw (-3,-1) node{$\sMG(1)$};

	\draw[thick]  (6,1) arc (0:360:2 and 1);
	\draw (4,1) node{$\mathsf{FP\,OFF}$};
	\draw[thick]  (11,1) arc (0:360:2 and 1);
	\draw (9,1) node{$\mathsf{FP\,ON}$};
	\draw (6.5,-1) node{$\sMG(2)$};
		
	\draw[thick]  (18,1) arc (0:360:2 and 1);
	\draw (16,1) node{$\mathsf{FP\,ON}$};
	\draw (16,-1) node{$\sMG(3)$};

	\draw(6,-2.5) node{(c)};
\end{tikzpicture}
\\
\
\\
\begin{tikzpicture}
[scale=0.4]
	\node at (0,1) [shape=rectangle,draw,label=below:$PN(1)$] {$0<l<u<\gamma\theta$};
	\node at (8,1) [shape=rectangle,draw,label=below:$PN(2)$] {$0<l<\gamma\theta <u$};
	\node at (16,1) [shape=rectangle,draw,label=below:$PN(3)$] {$0<\gamma\theta <l <u$};
	\draw[thick] (3.5,1) -- (4.5,1);
	\draw[thick] (11.5,1) -- (12.5,1);

	\draw(6,-2.5) node{(d)};
\end{tikzpicture}
\caption{(a) A self activating one node network. 
(b) Phase plane for switching network.  
(c) Annotated Morse graphs: $\sMG(1)$ has a single node generated by an attracting cell for which the fixed point is less than the threshold and annotated by $\mathsf{FP\,OFF}$; $\sMG(3)$ has a single node generated by an attracting cell for which the fixed point is greater than the threshold and annotated by $\mathsf{FP\,ON}$; and $\sMG(2)$ has two minimal nodes generated by  attracting cells in one of which the fixed point is less than the threshold and in the other the fixed point is greater than the threshold.
(d) Parameter graph.   
}
\label{fig:RN1}
\end{figure}
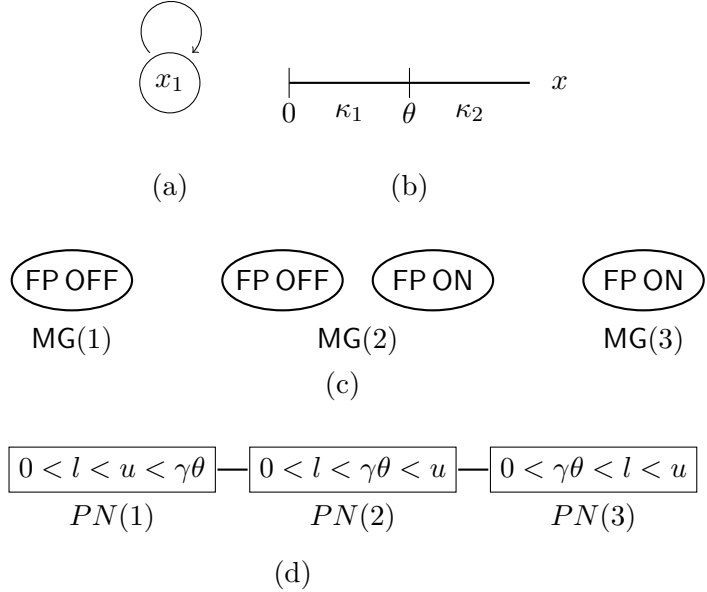

\begin{prop} \label{prop:GPGConnected} The $\GPG$ is a connected graph. \end{prop}

\begin{proof}
Our proof is based on finding a straight line path between points in any two parameter nodes from which we can prove the existence of a corresponding path in the $\GPG$. The burden is to show we can avoid pathologies where (a) the straight line path passes through an accumulation of parameter nodes, or (b) passes from one parameter node to another through a point which is not $1$-deficient. Since the parameter nodes are all in the strictly positive orthant, we only need to consider deficiencies arising from one of the following equalities being satisfied (the other equalities of \eqref{eq:equalities} do not intersect the strictly positive orthant and need not be considered):
\begin{align}
0 & = l_{j,i} - u_{j,i}, \nonumber \\
0 & = \theta_{j,i}-\theta_{j',i}\quad \text{for distinct $j,j'\in\Targets(i)$, or} \label{eq:positiveequalities} \\
0 & = \gamma_i\theta_{j,i} - \Lambda_i(\kappa) \quad\text{where $\theta_{j,i}$ defines a face of $\kappa$;} \nonumber
\end{align}

Each of the equalities in \eqref{eq:positiveequalities} has a solution set in $(0,\infty)^D$ (for $D:= N+3\cdot\#(E)$) which we show is a codimension-1 submanifold. To see this, let $f_1, f_2, \cdots, f_k : (0,\infty)^D \rightarrow \R$ be the functions such that the equalities of \eqref{eq:positiveequalities} may be written $f_i = 0$ for $i=1,2,\cdots,k$, where $k$ is the number of equalities in \eqref{eq:positiveequalities}. We leave it to the reader to inspect \eqref{eq:positiveequalities} and see that the gradient of each $f_i$ (for each $i=1,2,\cdots,k$) is non-vanishing, i.e. $0$ is a regular value of $f_i$. By the Regular Value Theorem, the varieties $f^{-1}_i(0)$ are thus codimension-1 submanifolds $M_i$ of $(0,\infty)^D$ for $i=1,2,\cdots,k$. 

We consider a collection of line segments $\mathcal{L} = \mathcal{L}_a \cup \mathcal{L}_b$ where $\mathcal{L}_a$ denotes the collection of line segments which intersect $\bigcup_{i=1}^k M_i$ infinitely often, and $\mathcal{L}_b$ denotes the collection of line segments which intersect $\bigcup_{i\neq j} M_i \cap M_j$ (i.e. any two of the submanifolds at the same point). Describing line segments by their endpoints we regard $\mathcal{L}$ as a subset of $(0,\infty)^{2D}$. We make the following technical claim: $\mathcal{L}$ is a set of zero measure. We show this by proving both $\mathcal{L}_a$ and $\mathcal{L}_b$ have zero measure. Note this corresponds to showing the pathologies (a) and (b) are rare.

To see that $\mathcal{L}_a$ has zero measure, observe first that the union of finitely many codimension-1 submanifolds $\bigcup_{i=1}^k M_i$ has measure zero. It follows that the collection of line segments with either endpoint in $\bigcup_{i=1}^k M_i$ has measure zero. For any line segment with endpoints $p$ and $q$ which are both not in $\bigcup_{i=1}^k M_i$, we consider its parameterization $\psi:[0,1] \rightarrow (0,\infty)^D$ defined by $\psi(t) := (1-t) p + t q$. The compositions $f_i(\psi(t))$ are polynomials in $t$ which are not identically zero. By the Fundamental Theorem of Algebra they each have a finite number of zeroes. Since the zeros correspond to the intersections of the segment with the submanifolds, the segments are not in $\mathcal{L}_a$. It follows that $\mathcal{L}_a$ has measure zero.

To see that $\mathcal{L}_b$ has zero measure, define $f_{ij}:(0,\infty)^D \rightarrow \R^2$ via $f_{ij}(x) := (f_i(x), f_j(x))$ for $i \neq j$. We leave to the reader to inspect \eqref{eq:positiveequalities} and verify that since for $i\neq j$ the set of variables appearing in the expressions for $f_i(x)$ and $f_j(x)$ are not the same, the gradients $\nabla f_i\vert_p$ and $\nabla f_j\vert_p$ are linearly independent for any $p \in (0,\infty)^D$, i.e. the Jacobian $Df_{ij}\vert_p$ has full rank. Thus $0$ is a regular value of $f_{ij}$. By the Regular Value Theorem, $M_i \cap M_j = f_{ij}^{-1}(0)$ is a codimension-$2$ submanifold of $(0,\infty)^D$. It follows that the set of line segments intersecting $M_i \cap M_j$ for some $i \neq j$ is measure zero. Hence $\mathcal{L}_b$ has measure zero.

Consequently, we may always perturb a line segment by making arbitrarily small adjustments to the location of its endpoints so that it intersects the submanifolds defined by \eqref{eq:equalities} finitely often, and never intersects two submanifolds at the same point. Now choose $\zeta, \zeta' \in \GPG$ and let $p:[0,1] \rightarrow \bar{Z}$ be a straight-line path such that $p(0) \in \zeta$ and $p(1) \in \zeta'$. Since $\zeta$ and $\zeta'$ are open, there exist neighborhoods in which we may perturb $p(0) \in \zeta$ and $p(1) \in \zeta'$; we use this freedom so we can assume without loss that the line segment $p([0,1])$ satisfies this technical property. By the first technical condition $p(t) \in Z$ for all but finitely many $t \in [0,1]$ so there exists a finite sequence $(\zeta_0 = \zeta, \zeta_1, \zeta_2, \cdots, \zeta_n = \zeta')$ through $\GPG$ corresponding to a finite sequence of intervals $[t_0 = 0,t_1), (t_1, t_2), \cdots, (t_{n}, t_{n+1}=1]$ in $[0,1]$ for which $p((t_i, t_{i+1}))= \zeta_n$. By the second technical condition, $p(t)$ only intersects one submanifold at a time, hence for each $i=1, \cdots, n$ the point $p(t_i)$ is a 1-deficient point in the intersection of the closures of $\zeta_{i-1}$ and $\zeta_i$. These facts together yield a path $\zeta = \zeta_0 \to \zeta_1 \to \zeta_2 \to \cdots \to \zeta_n = \zeta'$. Hence $\GPG$ is connected. 
\end{proof}

\subsection{Combinatorial Parameter Graph} \label{subsec:combinatorialparametergraph}

In this section we show how to assign to each parameter $z \in Z$ a combinatorial description $\phi$ which is sufficient to construct the wall-labeling (and hence state transition diagrams) induced by $z$. 
We call this combinatorial description a \emph{combinatorial parameter} and we denote the collection of combinatorial parameters by $\Phi$. Importantly, we show that for every parameter node $\zeta \in \cZ$, for any $z, z' \in \zeta$, both $z$ and $z'$ have the same associated combinatorial parameter $\phi \in \Phi$.

Our definition of combinatorial parameters is unfortunately somewhat technical. Overall it amounts to a bookkeeping system to keep track of the directions of the inequalities involving parameters which determine the threshold order and wall-labeling function. The threshold orderings give rise to inequalities comparing $\theta$ parameters. To determine the wall-labeling function requires comparing the various $\Lambda_i$ to $\gamma_i \theta_{j,i}$ (which gives the signs in \eqref{eqn:sign}). For a given regulatory network there are a fixed number of such comparisons required; our definition of a combinatorial parameter provides an organizational framework to speak of this collection of inequalities in a rigorous way.

\begin{defn} Define the \emph{input combinations} of the node $x_i$ to be the cartesian product 
\[ \mathsf{In}_j := \prod_{i \in \mathsf{\Sources}(j)} \{\text{off}, \text{on}\}.\]
Define the \emph{indicator function} $\chi_j : (0,\infty)^N \rightarrow \mathsf{In}_j$ such that \[ \chi_{j,i}(x) = 
\begin{cases}
\text{off} & \text{if $i \to j$ and $x_i < \theta_{j,i}$ or if $i \dashv j$ and $x_i > \theta_{j,i}$ } \\
\text{on} & \text{if $i \to j$ and $x_i > \theta_{j,i}$ or if $i \dashv j$ and $x_i < \theta_{j,i}$ } \\
\text{undefined} & \text{otherwise.}
\end{cases}\]
Define the \emph{valuation function} $v_j : \mathsf{In}_j \to \R^{\Sources{j}}$ via
\[ v_{j,i}(A) = 
\begin{cases}
l_{j,i} & \text{whenever } A_i =\text{off} \\
u_{j,i} & \text{whenever } A_i = \text{on} \\
\text{undefined} & \text{otherwise.}
\end{cases}
\]
Note that $\sigma_j = v_j \circ \chi_j$. 

Define the \emph{output combinations} of the node $x_i$ to be the set
\[ 
\mathsf{Out}_i := \mathsf{\Targets}(i).
\]

\end{defn}

\begin{defn} \label{defn:combparam} 
A \emph{logic parameter} is a function 
\[ L : \bigsqcup_{i=1}^N \left(\mathsf{In}_i \times \mathsf{Out}_i\right) \rightarrow \{-1, 1\}.\] We denote the restriction of $L$ onto $\mathsf{In}_i \times \mathsf{Out}_i$ as $L_i$. An \emph{order parameter} $O$ is a collection of total orderings $O_i$ of $\Targets(i)$ for each $i \in X$. A \emph{combinatorial parameter} is a pair $\phi = (L, O)$ where $L$ is a logic parameter and $O$ is an order parameter. We denote the collection of combinatorial parameters as $\Phi$. 

The \emph{combinatorial assignment function} $\omega : Z \rightarrow \Phi$ is given by 
$\omega(z) := (L, O)$ where
\begin{equation} \label{eqn:phi_sign} 
L_i(A, B) = \sgn\left( M_i \circ v_i (A) - \gamma_i\theta_{B, i}\right) \mbox{ for all $1 \leq i \leq N$  }
\end{equation}
and $O = O(z)$. 
For all $z \in Z$, we say that $\omega(z)$ is \emph{the combinatorial parameter associated to the parameter $z$}. 

The \emph{parameter region associated with the combinatorial parameter $\phi$} is given by $\omega^{-1}(\phi) \subset Z$ and denoted by $\vert \phi \vert$.
A combinatorial parameter $\phi \in \Phi$ is \emph{realizable} if there exists $z \in Z$ such that $\phi = \omega(z)$.
\end{defn}

\begin{defn} 
Let $\phi = (L, O) \in \Phi$ be a combinatorial parameter. 
Suppose $z \in Z$ such that $O = O(z)$. 
We may induce a wall-labeling on $\mathcal{W}(z)$ as follows. 
Let $(\tau, \kappa)$ be a wall with projection index $i$ and switching index $j$. We say $(\tau, \kappa)$ is \emph{an absorbing wall} with respect to $\phi$ if $L_i(\chi_i(\kappa),j) = -\sgn((\tau,\kappa))$ and \emph{an entrance wall} if $L_i(\chi_i(\kappa),j) = \sgn((\tau,\kappa))$.
\end{defn}

The next result shows we have successfully given a combinatorial description $\phi = \omega(z)$ for each $z \in Z$.

\begin{prop} Let $z \in Z$. Define $\phi \in \Phi$ such that $\phi = \omega(z)$. Then the wall-labeling induced by $z$ and the wall-labeling induced by $\phi$ are the same. \end{prop}
\begin{proof}
It suffices to show $\phi$ and $z$ induce the same wall-labeling, which then in turn will induce the same wall graph. To this end it suffices to show for each wall $(\tau, \kappa)$ with projection index $i$ and switching index $j$ that (1) $L_i(\chi_i(\kappa),j) = +1$ is equivalent to $M_i(\sigma_i(\kappa))> \gamma_i \theta_{j,i}$, and (2)
$L_i(\sigma_i(\kappa),j) = -1$ is equivalent to $M_i(\sigma_i(\kappa)) < \gamma_i \theta_{j,i}$. This follows from the definitions and  (\ref{eqn:phi_sign}).
\end{proof}

We supplement combinatorial parameters with a notion of adjacency. To understand this notion of adjacency it helps to remember that a combinatorial parameter is nothing more than a bookkeeping system giving the direction of the inequality for a set of inequalities describing a region of parameter space. Our notion of adjacency corresponds to reversing a single one of these inequalities.

\begin{defn} \label{defn:paramnodeadj} 
Let $\phi = (L, O)$, $\phi' = (L', O') \in \Phi$ be distinct combinatorial parameters. Denote the domains of $L$ and $L'$ by $\mathcal{D} := \bigsqcup_{i=1}^N \left(\mathsf{In}_i \times \mathsf{Out}_i\right)$. We say $\phi$ and $\phi'$ are \emph{adjacent} if either (1) $O = O'$ and there exists $x \in \mathcal{D}$ such that $\phi$ and $\phi'$ are equal on all of $\mathcal{D}$ except $x$, or (2) $L = L'$, $O_i = O'_i$ for $i \in X \setminus {i^*}$, and the total orders $O_{i^*}$ and $O'_{i^*}$ differ only by a single swap of the ordering of consecutive thresholds.
 \end{defn}

The previous definition renders $\Phi$ into a large graph of combinatorial parameters. Since not every combinatorial parameter is realizable (i.e. in the image of $\omega$) we are interested only in a subgraph.

\begin{defn} \label{defn:combinatorialparametergraph} 
The \emph{combinatorial parameter graph} $\CPG$ is the undirected graph on the realizable combinatorial parameters with an edge between two parameter nodes $\phi$ and $\phi'$ if and only if they are adjacent in the sense of Definition \ref{defn:paramnodeadj}.
 \end{defn}

\begin{ex}
\label{ex:RN1}
To illustrate the idea of the combinatorial parameter graph, we  return to one node regulatory network, indicated in Figure~\ref{fig:RN1}(a), of Example~\ref{ex:hyparrange}.
Recall that the switching network is \eqref{eq:RN1};  the phase space is indicated in Figure~\ref{fig:RN1}(b); and Figure~\ref{fig:RN1}(c) indicates the annotated Morse graphs.
For the purposes of this example it is instructive to construct the collection of combinatorial parameters $\Phi$ and see which ones are realizable.
There is only a single node, so
\[
\mathsf{In} = \setof{ \text{off}, \text{on} }\quad\text{and}\quad \mathsf{Out} = \setof{1}
\]
and the order parameter is similarly trivial $O = \setof{\theta}$.
The set of logic parameters are
\begin{align}
\quad (\text{off},1) \mapsto -1 \quad & \text{and}\quad (\text{on},1) \mapsto -1 \label{eq:logic1}\\
\quad(\text{off},1) \mapsto 1 \quad & \text{and}\quad (\text{on},1) \mapsto -1 \label{eq:logic2}\\
\quad(\text{off},1) \mapsto -1 \quad & \text{and}\quad (\text{on},1) \mapsto 1 \label{eq:logic3}\\
\quad(\text{off},1) \mapsto 1 \quad & \text{and}\quad (\text{on},1) \mapsto 1 \label{eq:logic4}
\end{align}
and therefore $\Phi$ contains four elements.

To determine the realizable combinatorial parameters we note that the indicator function $\chi\colon (0,\infty)\to \mathsf{In}$ is
\[
\chi(x) = \begin{cases} \text{off} & \text{if $x<\theta$} \\ 
\text{on} & \text{if $\theta< x$} \\ 
\text{undefined} & \text{otherwise} \\
\end{cases}
\]
and the valuation function $v : \mathsf{In} \to \R$ is 
\[
v(A) = \begin{cases} l & \text{if $A =$ off} \\ 
u & \text{if $A =$ on.}  \\
\end{cases}
\]
Note that $M$ is the identity and thus to understand the image of the combinatorial assignment function $\omega : Z \rightarrow \Phi$ we only need to consider 
\[
L(A,B) = \sgn\left(  v (A) - \gamma \theta \right).
\]
Observe that if $A=$ off, then 
\[
L(\text{off},1) = \sgn\left(  l - \gamma \theta \right) = \begin{cases} -1 & \text{if $l<\gamma\theta$} \\
1 & \text{if $\gamma\theta < l$} \\
\end{cases}
\]
and if $A=$ on, then 
\[
L(\text{on},1) = \sgn\left(  u - \gamma \theta \right) = \begin{cases} -1 & \text{if $u<\gamma\theta$} \\
1 & \text{if $\gamma\theta < u$} \\
\end{cases}
\]
Observe that if $l<u<\gamma\theta$, $l<\gamma\theta < u$, and $\gamma\theta < l<u$, then we obtain the combinatorial parameters with logic parameters \eqref{eq:logic1},
\eqref{eq:logic3}, and \eqref{eq:logic4}, respectively. Notice that the combinatorial parameter with logic parameter \eqref{eq:logic2} is not realizable since it would require $\gamma\theta < l$ and $u < \gamma \theta$, which contradicts $l < u$.
Thus $\CPG$ has three nodes.
The edges are as indicated in Figure~\ref{fig:RN1}(d), since these correspond to the switching of a single inequality. 
Thus, as promised (see Theorem~\ref{thm:isomorphism}) the $\CPG$ agrees with the $\GPG$ of Example~\ref{ex:hyparrange}.
\end{ex}

\subsection{Product Structure of Parameter Graph} \label{subsec:factorization}

Both $\GPG$ and $\CPG$ have a product structure. To describe this structure we need to define what is meant by a product of graphs:
\begin{defn} \label{defn:graphproduct} Given a collection of graphs $\{G_\alpha\}_{\alpha\in J}$ the \emph{graph product} $\prod_{\alpha \in J} G_\alpha$ is the graph with nodes which are $J$-tuples $x$ such that $x_\alpha \in G_\alpha$ for $\alpha \in J$ and two $J$-tuples $x$ and $y$ are adjacent if and only if $x_\alpha = y_\alpha$ for all but possibly one exceptional value $\alpha^* \in J$, and for that exceptional $\alpha^*$, $x_{\alpha^*}$ and $y_{\alpha^*}$ are adjacent in $G_{\alpha^*}$. \end{defn}

\begin{defn} For each $i \in X$, define the \emph{geometric factor graph} $\GPG_i$ to be the connected components of the complement of the solutions of the equalities
\begin{align*}
0 & = l_{j,i} - u_{j,i}, \\
0 & = \gamma_i, 0 = \theta_{j,i}, 0 = l_{j,i}, 0 = u_{j,i}, \\
0 & = \theta_{j,i}-\theta_{j',i}\quad \text{for distinct $j,j'\in\Targets(i)$, or} \\
0 & = \gamma_i\theta_{j,i} - \Lambda_i(\kappa) \quad\text{where $\theta_{j,i}$ defines a face of $\kappa$;}
\end{align*}
in \[ Z_i := \{ (l,u,\theta,\gamma) \in [0,\infty)^{1 + 3\#\Targets(i)} \}.\]
Two connected components are considered to be adjacent if they admit a $1$-deficient point in the intersection of their closures.
\end{defn}

\begin{defn}\label{def:combinatorialfactorgraph}
Let $i \in X$. Define $\Phi_i$ to be the collection of pairs $(L_i, O_i)$ where $O_i$ is an ordering of $\Targets(i)$ and $L_i$ is a function
\[ L_i : \mathsf{In}_i \times \mathsf{Out}_i \rightarrow \{-1, 1\}.\] Define a function $\omega_i : Z_i \rightarrow \Phi_i$ via $\omega(z) := \phi_i = (L_i, O_i)$ whenever
\begin{equation} \label{eqn:cpg_factor}
L_i(A, B) = \sgn\left( M_i \circ v_i (A) - \gamma_i\theta_{B, i}\right).
\end{equation}
We say two elements $\phi_i, \phi_i' \in\Phi_i$ are \emph{adjacent} if they differ in only one value. We say $\phi_i$ is \emph{realizable} if $\omega_i^{-1}(\phi_i)$ is non-empty. We denote the subgraph of realizable elements in $\Phi_i$ as $\CPG_i$ and call it the \emph{combinatorial factor graph.}
\end{defn}

\begin{thm} \label{thm:factorization} For either $\sPG = \GPG$ or $\sPG = \CPG$, we have the following \emph{factor decomposition}:
\[ \sPG = \prod_{i=1}^N \sPG_i. \]
\end{thm}

\begin{proof} We show it first for the geometric parameter graph. First we demonstrate a one-to-one correspondence between the vertices of $\prod \GPG_i$ and $\GPG$. Observe that $Z = \prod Z_i$. By straightforward topological arguments a connected component in $Z$ is a product of connected components of $Z_i$ and vice-versa. This establishes a one-to-one correspondence of vertices between $\prod \GPG_i$ and $\GPG$. Now we show that vertices $\zeta, \zeta'$ are adjacent in $\GPG$ if and only if  they are adjacent in $\prod \GPG_i$. Assume first that $\zeta$ and $\zeta'$ are adjacent in $\GPG$. Then there exists a $1$-deficient point in the intersection of their closure. At all such $1$-deficient points there is a single equality of \eqref{eq:equalities} which is satisfied; it involves parameters corresponding to a definite factor $\GPG_i$ for some $i$. For all but a single exceptional $i = i^*$ we have $\pi_i(\zeta) = \pi_i(\zeta')$ and in the exceptional case $\pi_{i^*}(\zeta)$ and $\pi_{i^*}(\zeta')$ are adjacent in $\GPG_{i^*}$. By Definition~\ref{defn:graphproduct} this means $\zeta$ and $\zeta'$ are adjacent in $\prod \GPG_i$. The converse more or less runs this argument in reverse. Hence $\GPG = \prod \GPG_i$.

Next we show $\CPG = \prod \CPG_i$. In this case it follows very immediately from the definitions; $\Phi = \prod \Phi_i$ and the definition of adjacency in Definition~\ref{defn:paramnodeadj} is compatible with Definition~\ref{defn:graphproduct}. What remains is to see that a combinatorial parameter $\phi$ is realizable only if $\phi_i$ is realizable for all $i \in X$; to this end we leave it to the reader to verify from the definitions that $\omega^{-1}(\phi) = \prod \omega_i^{-1}(\phi_i)$ from which the result follows.
\end{proof}

The utility of this decomposition theorem is that it allows us a way of understanding parameter space piecewise; given a network we can consider a single node $i$ that has $n = \#\Sources(i)$ inputs, $m = \#\Targets(i)$ outputs, and an associated logic function $M_i$. Given these three things we may construct the factor graph. Thus, we can store a library of such factor graphs once and for all, and given a network we can immediately understand the combinatorial decomposition of parameter space by assembling this product structure.

\subsection{Computation of Parameter Graph} \label{subsec:computation}

Theorem \ref{thm:factorization} allows us to construct parameter graphs as products of factor graphs, so the problem of computing a parameter graph reduces to the problem of constructing the factor graphs.  In light of Proposition \ref{prop:GPGConnected} the factor graph is connected so we may explore it, and hence construct it, via any graph traversal technique, e.g.\ breadth-first-search, depth-first-search. 
To use this approach we require two ingredients: we must have a starting point, and we must be able construct the list of adjacent parameters.

For a starting point we may choose the combinatorial parameter $\phi=(L,O)$ where 
$L \equiv -1$ and $O$ may be chosen freely. This is guaranteed to be realizable: we obtain a realization if we choose $u$, $l$, and $\gamma$ values freely and then choose $\theta$ values in the desired ordering and sufficiently large so that $\Lambda_i < \gamma_i \theta_{j,i}$ for all $(i,j) \in E$.

To compute adjacency lists we first obtain from Definition~\ref{defn:paramnodeadj} a list of candidate adjacent combinatorial parameters in $\Phi$. Not every candidate adjacency is in $\CPG$ since not every combinatorial parameter is realizable. To obtain the adjacencies we filter the list of candidates in search of the realizable combinatorial parameters. In particular, we employ the computational algebra technique known as \emph{cylindrical algebraic decomposition} (CAD) \cite{collins1998partial} which provides a finite, recursive description of the geometric region $\zeta'$ associated with each candidate adjacent combinatorial parameter $\phi'$ (i.e. $\zeta' = \omega^{-1}(\phi')$). Since $\zeta'$ is given in terms of strict inequalities, we use the algorithm of \cite{strzebonski2000solving} as implemented in Mathematica. 

These algorithms can be quite expensive (worst case bounds are doubly-exponential in the number of variables) but are tractable for networks built out of components given in Table~\ref{table:goodnodes}.
From a CAD description of $\zeta'$ we can determine if it is empty and hence if $\phi'$ is realizable. If it is then $\zeta'$ is added into list of adjacent vertices of $\CPG$.

In this manner we can traverse the graph and construct the factor graphs, and hence the graph. Note that given a network component there is an $\#\Targets(i)!$-fold symmetry due to the number of possible rearrangements of the $\theta$ variable orderings; this can be used to reduce the number of CAD calculations; in particular we can compute the parameter graph only for a single ordering of thresholds and from this we can extrapolate the rest of the factor graph via symmetry in a straightforward manner. In particular, we can make a disjoint union of $\#\Targets(i)!$ independent copies of the graph and then connect them together via adjacencies corresponding to swapping threshold orderings whenever none of the $\Lambda$ values are between them.

\begin{rem}
\label{rem:CAD} As noted above  CAD computations can be quite expensive. 
Currently we know of no general way of constructing a $\CPG$ graph without using CAD unless we are content with performing computations for combinatorial parameters which are not realizable, i.e.\ do not correspond to some actual parameter $z \in Z$. 
This may seem like a defect of this approach but we should emphasize that  these CAD computations do not need to be repeated for each network we analyze. 
Rather, we construct a library of CAD results corresponding to each single node in a network with a given number of inputs, number of outputs, and logic. From this library of CAD computations we may analyze any network which may be built up from these components without doing any further computational algebraic geometry. 
Because this library has been constructed for all nodes of the form in Table \ref{table:goodnodes}, we may analyze any network built out of components with these node types without performing any additional CAD computations. 
The Mathematica scripts we executed to produce our library of CAD results may be found in the DSGRN software package \cite{dsgrn}.
\end{rem}

\subsection{Relationship between Geometric and Combinatorial Parameter Graphs} \label{subsec:equivalence}

For all $z, z' \in \zeta \in \GPG$, $\omega(z) = \omega(z')$, there is a well-defined map \[ h : \GPG \to \CPG \] such that $\zeta \mapsto \omega(\zeta)$. Moreover, the existence of a $1$-deficient point in the intersection of the closures of the connected components $\zeta, \zeta' \in \cZ$ immediately implies $h(\zeta)$ and $h(\zeta')$ are described by sets of inequalities which differ by only a single inequality reversal and hence are adjacent in $\CPG$. This renders $h$ into a \emph{graph homomorphism} as it maps vertices to vertices and edges to edges. Note that $h$ behaves nicely with the product structure indicated in the previous section.

As stated earlier, it is an open question whether or not $h$ is in general an isomorphism. The question ultimately comes down to two issues: (1) Are the geometric regions associated with the realizable combinatorial parameters connected? (2) Do the geometric regions associated with adjacent vertices of $\CPG$ always admit a $1$-deficient point in the intersection of their closures?
While we cannot at this time answer these questions in complete generality,  we do have the following result and a method for checking specific cases.

\begin{thm}\label{thm:isomorphism} The homomorphism $h$ is an isomorphism for regulatory networks for which each node in the regulatory network is described by an entry of Table~\ref{table:goodnodes}.
\end{thm}

\begin{proof} By Theorem \ref{thm:factorization} it suffices to consider only the factor graphs; all that we must show is that \[ \GPG_i \simeq \CPG_i.\]

This in turn becomes a finite computation for each entry of Table~\ref{table:goodnodes}.
In particular we use CAD to show the following:
\begin{enumerate}
\item For each $\phi_i \in \CPG_i$, the associated geometric region is connected.
\item For each pair of adjacent $\phi_i, \phi_i' \in \CPG_i$, the associated geometric regions admit a 1-deficient point in the intersection of their closures.
\end{enumerate}

For (1), observe each node $\phi_i$ in $CPG_i$ is associated with a parameter region $\vert \phi_i \vert$ in $Z$ defined by a collection of polynomial inequalities. These inequalities are obtained by writing threshold inequalities $\theta_{j_1,i} < \theta_{j_2,i}$ according to the threshold ordering $O_i$ and by writing either $M_i \circ v_i (A) < \gamma_i\theta_{B, i}$ or $M_i \circ v_i (A) > \gamma_i\theta_{B, i}$ for each $(A,B)\in \mathsf{In}_i \times \mathsf{Out}_i$ according to the sign of $L_i(A,B)$.

For (2), we consider the inequalities associated with adjacent parameters $\phi_i$ and with $\phi_i'$ in $CPG_i$. These two sets of inequalities are identical apart from a single inequality with its sign flipped. Dropping this inconsistent inequality yields a set of inequalities $\mathcal{I}$ describing a region in $\bar{Z}_i$ which contains the union of $\vert \phi_i \vert$ and $\vert \phi_i' \vert$. Since all the inequalities are strict, the solutions of $\mathcal{I}$ comprise an open set. If this open set is connected, then as an open subset of Euclidean space it is also path-connected; in particular take any path from a point in $\vert \phi_i \vert$ to a point in $\vert \phi_i' \vert$ in $\vert \mathcal{I} \vert$. By the intermediate value theorem it must pass through a point where equality is obtained on the inequality $\phi_i$ and $\phi_i'$ disagree on (but on none of the other inequalities). Thus we have identified a 1-deficient point on the intersection of the closures of $\vert \phi_i \vert$ and $\vert \phi_i' \vert$. 

Using Mathematica, we have performed these connectedness computations and verified (1) for each vertex and (2) for each adjacency of the factor graphs corresponding to network components indicated in the rows of Table~\ref{table:goodnodes}. 
The scripts which accomplish these computations may be found in the DSGRN software package \cite{dsgrn}.
\end{proof}

\begin{table}[]
\centering
\caption{Network Node Components}
\label{table:goodnodes}
\begin{tabular}{|c|c|c|c|}
\hline
  $\# \Sources(i)$ & $\#\Targets(i)$ & $M_i$ & $\# \sPG_i / \#\Targets(i)!$  \\
 \hline
 1 & 1 & $x$ & 3  \\
 1 & 2 & $x$ & 6  \\
 1 & 3 & $x$ & 10 \\ 
 2 & 1 & $x+y$ & 6  \\
 2 & 1 & $xy$ & 6  \\
 2 & 2 & $x+y$ & 20  \\
 2 & 2 & $xy$ & 20  \\
 2 & 3 & $x+y$ & 50  \\
 2 & 3 & $xy$ & 50  \\
 3 & 1 & $x + y + z$ &  20  \\
 3 & 1 & $xyz$ &  20  \\
 3 & 1 & $x(y+z)$ & 20  \\
 3 & 2 & $x + y + z$ &  150 \\
 3 & 2 & $xyz$ &  150 \\
 3 & 2 & $x(y+z)$ & 155 \\
 3 & 3 & $x + y + z$ &  707 \\
 3 & 3 & $xyz$ &  707 \\
 3 & 3 & $x(y+z)$ & 756 \\
 \hline
\end{tabular}
\end{table}

%
\section{Applications}
\label{sec:applications}

We demonstrate the use of the DSGRN database with two  elementary three-node networks, the repressilator and bistable repressilator models, and an example associated with the p53 network.
These databases, along with a growing number of other examples, can be found at the $\mathsf{DSGRN}$ database \cite{dsgrn}.

\subsection{The Repressilator}
\label{sec:repressilator}

We begin with the repressilator as shown in Figure~\ref{fig:networks}(a) for two reasons: 
first, it is of biological interest in that it has been constructed in \textit{E. coli} by~\cite{Elowitz00}; and second, it is extremely simple as it consists of  three genes that repress each other in a cycle and thus we can draw the entire parameter graph (see Figure~\ref{fig:pm}).

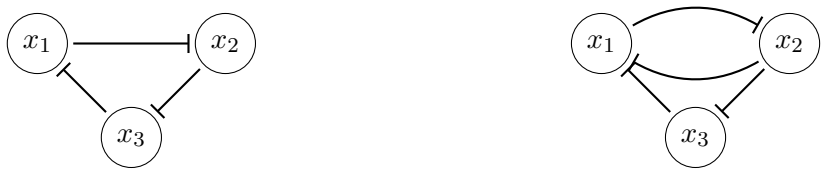
\begin{figure}[h!]
		\begin{center}
			\begin{tikzpicture}[main node/.style={circle,fill=white!20,draw,font=\sffamily\normalsize\bfseries},scale=2.5]
				\node[main node] (x1) at (0,0) {$x_1$};
				\node[main node] (y1) at (1,0) {$x_2$};
				\node[main node] (z1) at (0.5,-0.5){$x_3$};
				\node[main node] (x2) at (3,0) {$x_1$};
				\node[main node] (y2) at (4,0) {$x_2$};
				\node[main node] (z2) at (3.5,-0.5){$x_3$};

				\path[->,>=angle 90,thick]
				(x1) edge[-|,shorten <= 2pt, shorten >= 2pt] node[] {} (y1)
				(y1) edge[-|,shorten <= 2pt, shorten >= 2pt] node[] {} (z1)
				(z1) edge[-|,shorten <= 2pt, shorten >= 2pt] node[] {} (x1)
				;

				\path[->,>=angle 90,thick]
				(x2) edge[-|,shorten <= 2pt, shorten >= 2pt, bend left] node[] {} (y2)
				(y2) edge[-|,shorten <= 2pt, shorten >= 2pt] node[] {} (z2)
				(y2) edge[-|,shorten <= 2pt, shorten >= 2pt, bend left] node[] {} (x2)
				(z2) edge[-|,shorten <= 2pt, shorten >= 2pt] node[] {} (x2)
				;

			\end{tikzpicture}
			\caption{Left: Repressilator. Right: Bistable repressilator.}\label{fig:networks}
		\end{center}
\end{figure}

Applying Definition~\ref{defn:switching} the switching equations for the repressilator take the form
\begin{eqnarray} \label{repressilator}
\dot{x}_1 &=& -\gamma_1x_1 + \sigma^-_{1,3}(x_3) \nonumber \\
\dot{x}_2 &=& -\gamma_2 x_2 + \sigma^-_{2,1}(x_1) \\
\dot{x}_3 &=& -\gamma_3 x_3 + \sigma^-_{3,2}(x_2) \nonumber,
\end{eqnarray}
where the notation for the nonlinearities $\sigma$ follows the notation  of \eqref{eq:step_functions}.  
We write $\sigma^-$ as a reminder that the interaction represents repression; the second choice in \eqref{eq:step_functions}. 

The repressilator model has a single  regulatory threshold for each variable, dividing the phase space (which is the positive orthant of $\R^3$) into 8 cells and 12 walls. The domain and wall graphs from Section~\ref{subsec:diagrams} vary with parameter choice and therefore we  start with discussion of the space of parameters and  the parameter graph. 

For each $i$  the function $\sigma^-_{i+1,i}$ representing the edge $i \to {i+1}$ is parameterized  by  three parameters $u_{i+1,i},l_{i+1,i}$ and $\theta_{i+1,i}$ ($i$ is taken mod 3).  
In addition there are three decay constants $\gamma_1, \gamma_2, \gamma_3$. 
Therefore the parameter space $Z\subset \bar{Z}\subset \R^{12}$. 
To determine the parameter graph we first construct the combinatorial parameter graph $\CPG$ using Factor Decomposition Theorem~\ref{thm:factorization}. 
For each factor of the $\CPG$ we use the Table~\ref{table:goodnodes}.
Since the repressilator model has a single threshold for each  variable,  the threshold orders $O_i$ as in Definition~\ref{def:combinatorialfactorgraph} are trivial. 
The choice of logic $M_i$, see~\eqref{eq:lambdaAsComposition}, is also trivial. 
Therefore for each  $i=1,2,3$ we consider the  network node component in the  first row in Table~\ref{table:goodnodes}.  
The last column of the Table shows that  $ \# PG_i = 3$  for each variable $x_i$. 
To make this explicit, for variable $x_2$ these choices correspond to 
\[ 
l_{2,1} < u_{2,1} <  \gamma_1 \theta_{2,1},\quad l_{2,1} < \gamma_1\theta_{2,1} < u_{2,1},\quad \text{and}\quad \gamma_1\theta_{2,1} < l_{2,1} < u_{2,1}. 
\]
 By the Factor Decomposition Theorem~\ref{thm:factorization}, the combinatorial parameter graph $\CPG$ has $3^3=27$ nodes and edges as shown in Figure~\ref{fig:pm}(left).
 Finally, by Theorem~\ref{thm:isomorphism}  the combinatorial and geometrical parameter graphs are the same and therefore Figure~\ref{fig:pm}(left) also indicates the geometric parameter graph $\GPG$.
 
\begin{figure}
\centering
\begin{tabular}{ m{.5\textwidth}  m{.5\textwidth} }
\includegraphics[width=3in]{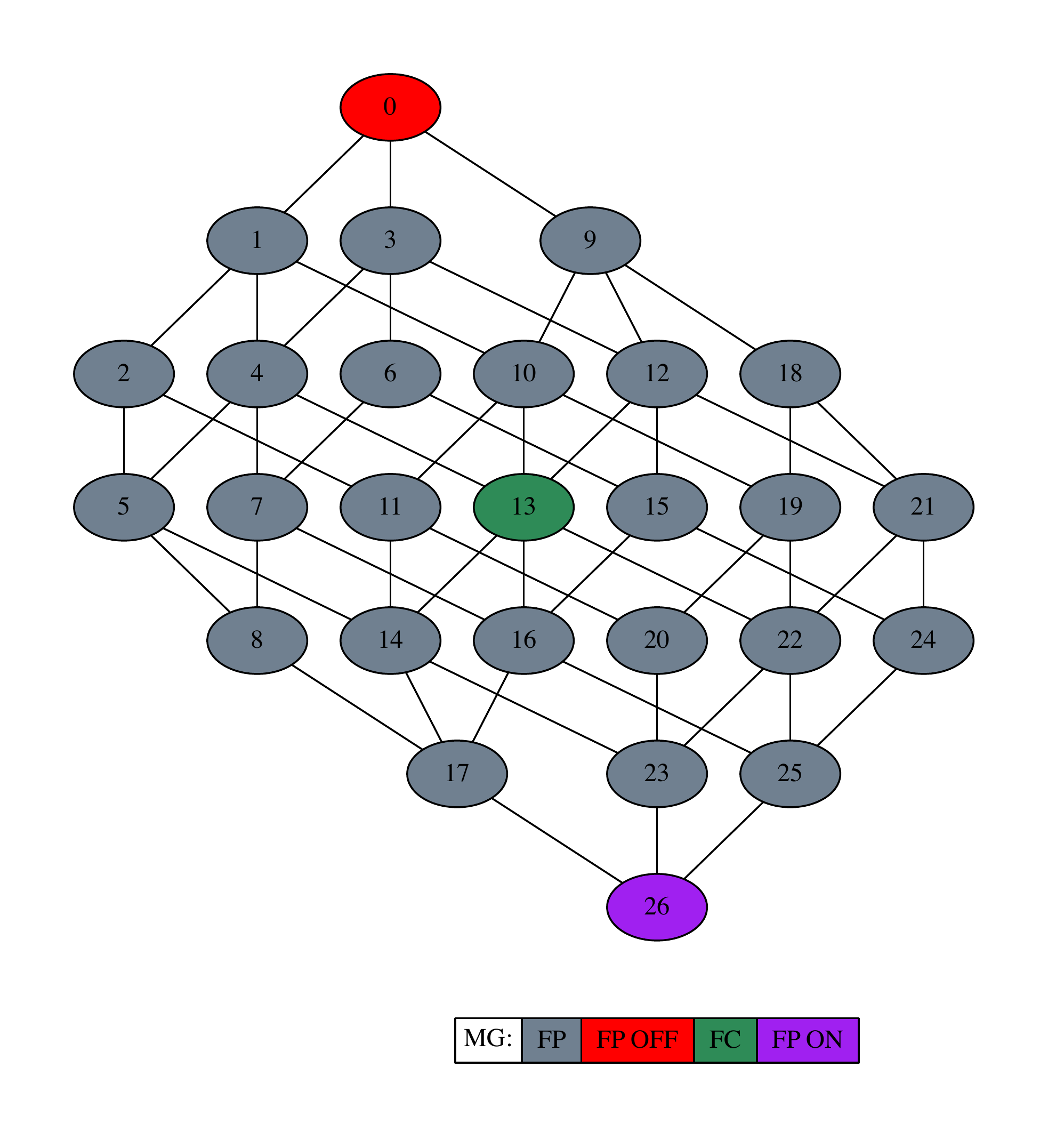} & \includegraphics[width=3.5in]{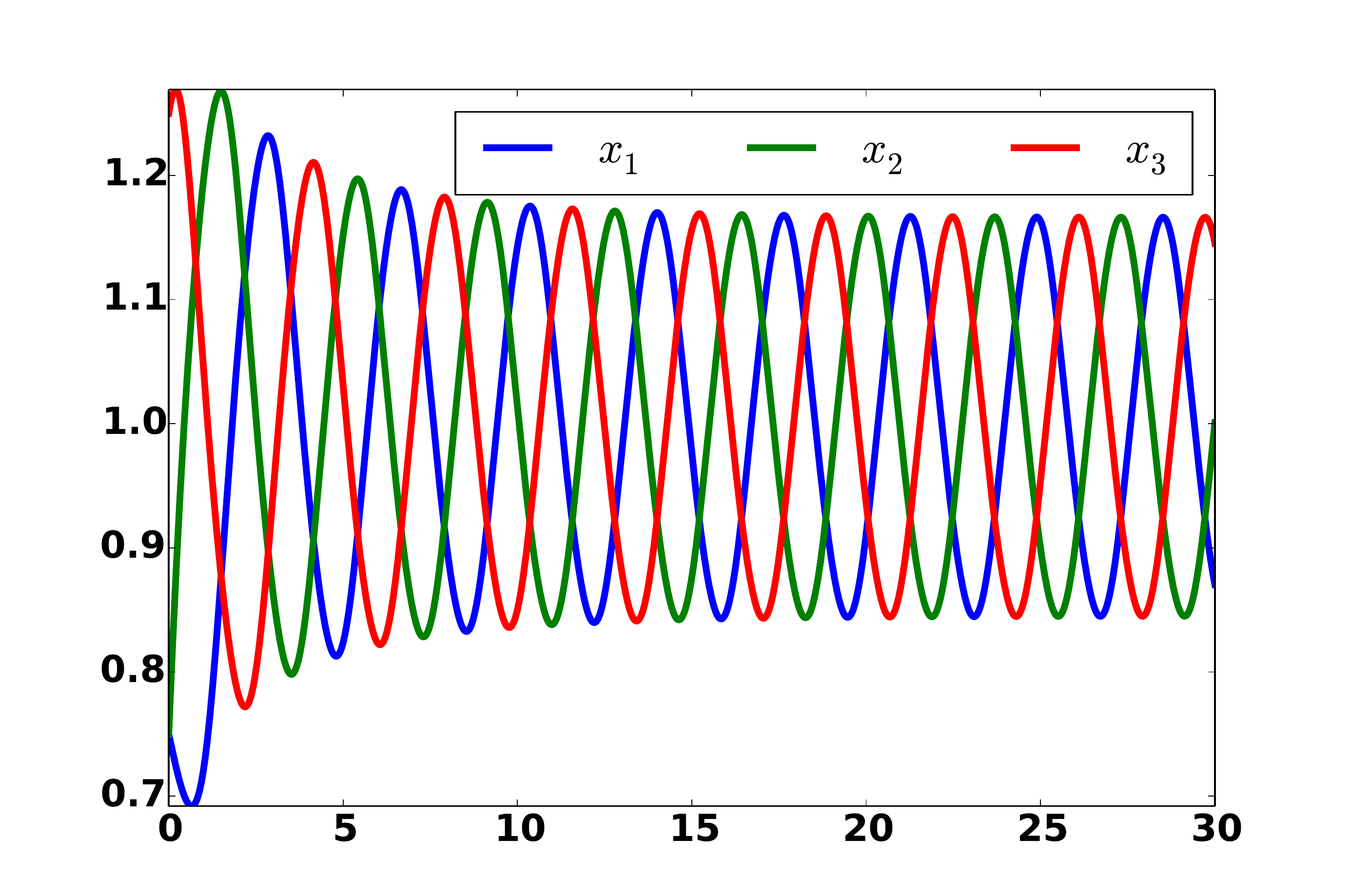}
\end{tabular}
\caption{Left: Repressilator parameter graph with same-colored parameter nodes corresponding to the same Morse graph. Right: Hill function simulation for the repressilator  satisfying the inequalities of parameter node 13 with $l_{i,j}=0.5$, $\theta_{i,j}=1.0$, $u_{i,j}=1.5$ (see Equation~\eqref{eq:rephill}). The Hill exponent is $n=9$.}
\label{fig:pm}
\end{figure} 
 
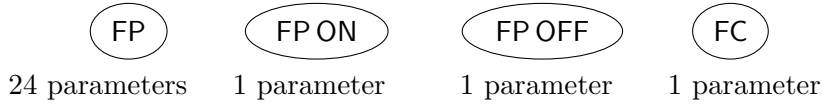
\begin{figure}[h!]
	\begin{tikzpicture}[main node/.style={ellipse,fill=white!20,draw,scale=1}]
		\node[main node] (FP) at (-4,0) {$\FP$};
		\node[main node] (FPON) at (-1.5,0) {$\FPon$};
		\node[main node] (FPOFF) at (1.5,0) {$\FPoff$};
		\node[main node] (FC) at (4,0) {$\FC$};

	\end{tikzpicture}
	\begin{tabular}{cccc}
		\hspace{-.1in}24 parameters & \hspace{.1in}1 parameter & \hspace{.2in} 1 parameter & \hspace{.1in} 1 parameter
	\end{tabular}
	\caption{$\mathsf{DSGRN}$ Morse graphs for the repressilator.}
\label{fig:mg_repress}	
\end{figure}

As indicated in Section~\ref{subsec:diagrams} the domain and wall graphs  and hence the the Morse graph (Definition~\ref{defn:morsegraph}) are potentially different at each of the 27 parameter nodes.
However, for the repressilator there is a single Morse graph each consisting of a single node.
This emphasizes the value of annotating the nodes of the Morse graph.
As described in Section~\ref{sec:annontation} this annotation indicates whether or not the
variable transition in the associated Morse set and if not whether their value is high or low.
As shown in Figure~\ref{fig:mg_repress} in the repressilator  the Morse graph can assume four different annotations: a fixed point, $\FP$; a fixed point where all variables are above the threshold, $\FPon$; a fixed point where all variables are below the threshold, $\FPoff$; and a full cycle $\FC$ where all variables pass  through their threshold.
In Figure~\ref{fig:mg_repress}, the number of parameters associated to each Morse graph is listed, with 24 of the 27 exhibiting the Morse graph $\FP$.

Returning to Figure~\ref{fig:pm} (left), which shows the entire parameter graph for the repressilator, each node is color-coded according to the associated  Morse graph.
We note that a single parameter node gives rise to this Morse graph $\FC$ indicating a periodic orbit. 
To further investigate the parameter set $R_{\FC} \subset Z$ which is represented by  this parameter node, the $\mathsf{DSGRN}$ database provides us with  the set of inequalities that define $R_{\FC}$:
\begin{eqnarray} \label{ineq}
l_{1,3} <  \gamma_1 \theta_{2,1} < u_{1,3} \nonumber \\
l_{2,1} <  \gamma_2 \theta_{3,2} <u_{2,1} \\
l_{3,2} <  \gamma_3 \theta_{1,3} <u_{3,2}. \nonumber
\end{eqnarray}

While  this is unique combination of parameters  is represented by a single parameter graph node,  $R_{\FC}$ is  clearly a substantial and unbounded component in  the parameter space $Z \subset \R^{12}$.   
To see if $R_{\FC}$ predicts well oscillations for a smooth repressilator model, we  replace the switching model by a model that uses  Hill function nonlinearities, which are closely related to the switching nonlinearities. 
Observe that for each $z\in \supp{\zeta}$, where $\zeta$ is a node in the parameter graph,  there is a natural one parameter family of Hill functions. For an activating step function $\sigma^+(x)$ this takes the form
\begin{equation} \label{eq:hillplus}
h_n^+(x) = l + (u-l) \frac{x^n}{\theta^n + x^n},  
\end{equation}
and the repressing step function $\sigma^-(x)$ by a Hill function $h_n^-(x) $
\begin{equation} \label{eq:hillminus}
 h_n^-(x) = l+ (u-l) \frac{\theta^n}{\theta^n + x^n}
 \end{equation}
 where the undetermined parameter is the Hill exponent $n$.
 Note that if $n \to \infty$, then $\lim_{n \to \infty} h_n^\pm(x) \to \sigma^\pm(x)$ pointwise for all $x \not = \theta$. 

We sample a point from $R_{\FC}$ that satisfies the inequalities (\ref{ineq}), select a Hill exponent, and simulate a Hill function model with these parameters. 
\begin{align} 
 \dot{x}_1 &= -x_1 + 0.5+ \frac{1}{1 + x_3^n}= -x_1 + g_3(x_3)  \nonumber\\
 \dot{x}_2 &= -x_2 + 0.5+ \frac{1}{1 + x_1^n}= -x_2 + g_1(x_1) \label{eq:rephill}\\
 \dot{x}_3 &= -x_3 + 0.5+ \frac{1}{1 + x_2^n}= -x_3 + g_2(x_2) \nonumber
 \end{align}
The choice of $n$ will affect the dynamics, and so it is important to choose an $n$ large enough so that the Hill function model~\eqref{eq:rephill} is reasonably representative of the switching model~\eqref{eq:switching}. 
For the repressilator, there is analysis available that allows us to suggest the minimum allowable $n$, which we calculate below. 
Tyson and Othmer~\cite{Tyson1978} proved a necessary secant condition for stability of a global fixed point $E$ in an $I$-dimensional cyclic feedback  system (see also Thron~\cite{Thron1991}),  given by   
\begin{equation*}
\frac{ \mid g_1'(E)\cdots g_I'(E) \mid }{\gamma_1\cdots\gamma_I} < \sec\left(\frac{\pi}{I}\right)^I, 
 \end{equation*}
 where the Jacobian of the system is 
 \begin{equation*}
 \begin{pmatrix}
 	-\gamma_1 & 0 & 0  & \dots & 0 & g_I'(E) \\
 	g_1'(E) & -\gamma_2 &  0 & \dots & 0 & 0\\ 
 	0 & g_2'(E) & -\gamma_3 &  \dots & 0 & 0\\
 	\dots & & & & & \\
 	0 & 0 & 0 & \dots & g_{I-1}'(E) & -\gamma_I
 \end{pmatrix} .
 \end{equation*}
 This condition is sharp when all of the decay rates are equal which is the case here.

Applying this condition to \eqref{eq:rephill}, we see that for the equilibrium is $E=(1,1,1)$ the secant formula becomes 
\begin{equation} \label{secant}
 \mid g_1'(1)g_2'(1)g_3'(1) \mid < \sec\left(\frac{\pi}{3}\right)^3.
 \end{equation}
 It is readily verified that $g_i'(1) = -n/4$ so that (\ref{secant}) takes the form
\[ n^3 < 4^3 2^3\] 
and therefore  $E=(1,1,1)$ is stable when $n < 8$. Since the condition is sharp, it is easy to show that at $n=8$ there is a Hopf bifurcation at which the equilibrium destabilizes and a stable periodic orbit is born. So for $n > 8$ in \eqref{eq:rephill}, there is a stable periodic orbit. In Figure~\ref{fig:pm} (right) we show a periodic orbit for  $n=9$.

\subsection{ The Bistable Repressilator}
The bistable repressilator is slightly more complicated than the repressilator in that it has an additional negative feedback. 
Figure~\ref{fig:networks} represents these regulatory networks in graphical form. 
Because of the double feedback loop we expect that for appropriate parameter values this system may exhibit bistability, hence the name. 
The associated switching system is given by
\begin{eqnarray} \label{bistable}
\dot{x}_1 &=& -\gamma_1x_1 + \sigma^-_{1,2}(x_2)\sigma^-_{1,3}(x_3) \nonumber \\
\dot{x}_2 &=& -\gamma_2 x_2 + \sigma^-_{2,1}(x_1) \\
\dot{x}_3 &=& -\gamma_3 x_3 + \sigma^-_{3,2}(x_2) \nonumber.
\end{eqnarray}
The logic $M_i$ for each node is trivial except for the first equation, where we have chosen AND logic, i.e.\ the negative influences of $x_2$ and $x_3$ are multiplicative.

In the bistable repressilator the variable $x_2$ represses both $x_3$ and $x_1$, and so there are two choices for $O_2$, $\theta_{3,2} < \theta_{1,2}$ or $\theta_{1,2} < \theta_{3,2}$. 
 Because of the extra threshold in the bistable repressilator, there are $12$ cells and 20 walls dividing the phase space. 
 
 The parameter space $Z \subset \R^{15}$ and we seek to understand the geometric parameter graph $\GPG$ that represents arrangement of components of $Z$. To construct the combinatorial parameter graph $\CPG$ we again consult the Table~\ref{table:goodnodes}. 
 The variable $x_1$ has 2 inputs, 1 output and the logic is multiplication; this corresponds to row $6$ in the Table and hence $\#  PG_1=6$. The  variable $x_2$ has 1 input and 2 outputs and so it correspond to row two. However, since the last column in Table~\ref{table:goodnodes} lists $\# PG_i$ divided by all possible permutations of output variable thresholds, $\# PG_2 = 12$.
 Finally, the variable $x_3$ has one input and one output, which corresponds  to the first row and thus  $\# PG_1 = 3$.
 By the Factor Decomposition Theorem~\ref{thm:factorization},  the $\CPG$ has $6*12*3=216$ parameter nodes and  by Theorem~\ref{thm:isomorphism}  the $\CPG$ and geometrical parameter graphs  $\GPG$ are the same. Since the parameter graph is sizable we only show in Figure~\ref{fig:bi_pm_4} a half of the parameter graph that corresponds to one of the two orders of the thresholds  in $O_2$ ($\theta_{3,2} < \theta_{1,2}$).

 \begin{figure}[h!]
 \centering
 \begin{tabular}{c}
 \includegraphics[width=6in]{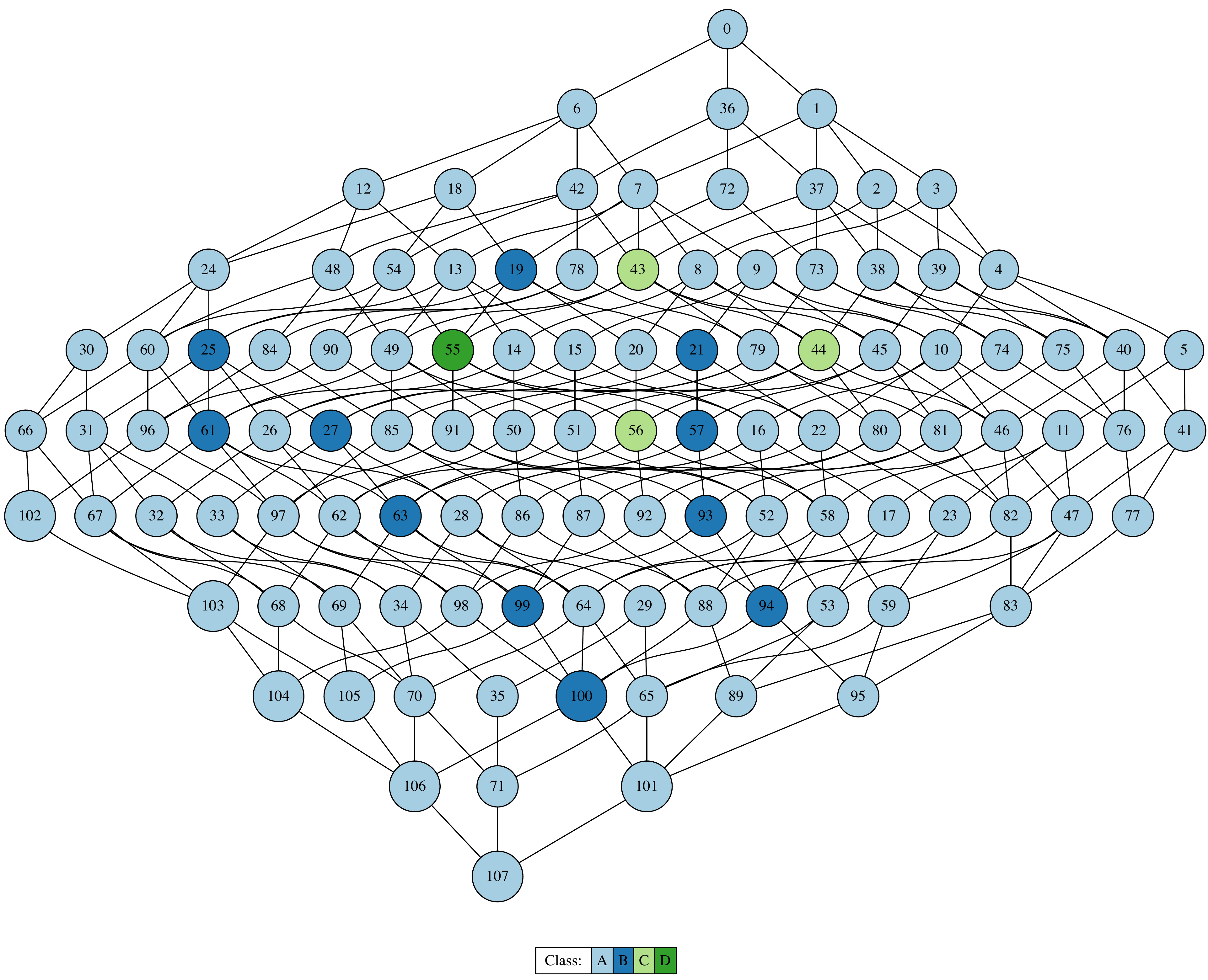} \\
 \includegraphics[width=3in]{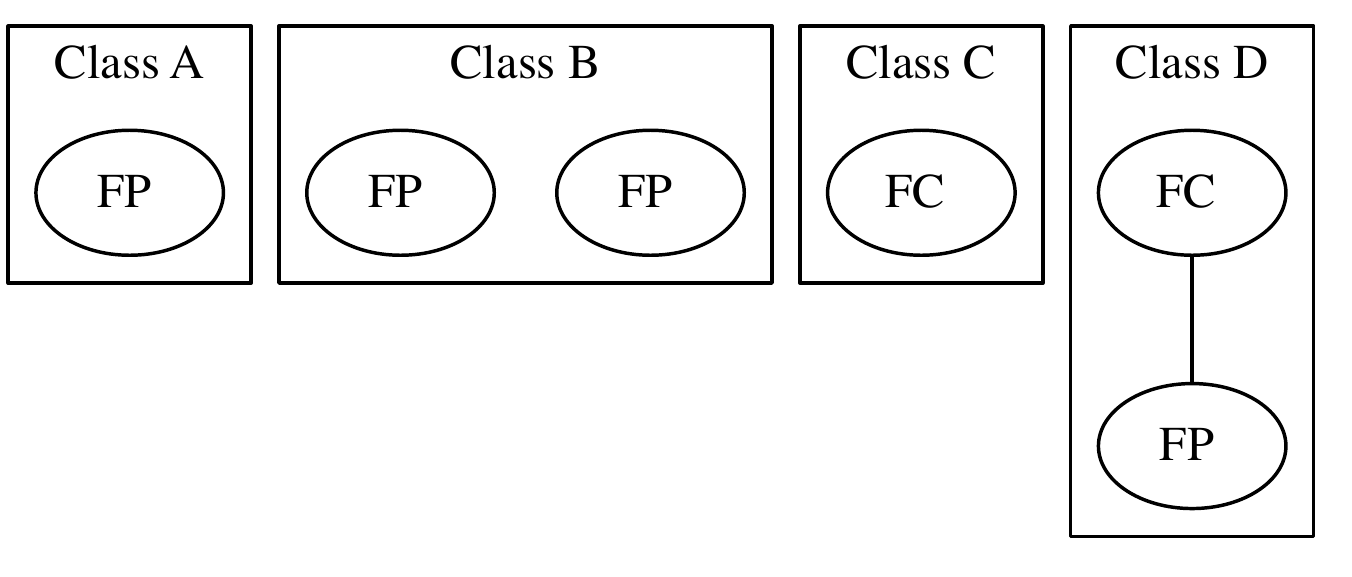}
 \end{tabular}
 \caption{Bistable repressilator parameter graph with colors corresponding to partitioned Morse graph continuation classes. Class A: single stable fixed points; Class B: bistability; Class C: stable cycle; Class D: unstable cycle with a stable fixed point.}
 \label{fig:bi_pm_4}
 \end{figure}

Using the same annotation as for the repressilator, there seven distinct classes indicated in Figure~\ref{fig:mg_bistable}.
For simplicity in Figure~\ref{fig:bi_pm_4} we group the parameter nodes based on the following four annotations:
type A nodes have Morse graph with  a single fixed point $\FP$; nodes in class B have a Morse graph with two fixed points $\FP$, and hence signal the presence of  bistability;   nodes of  type C have Morse graph $\FC$; and nodes of type D have the Morse graph with the  lower Morse set $\FP$ and the upper Morse set $\FC$.

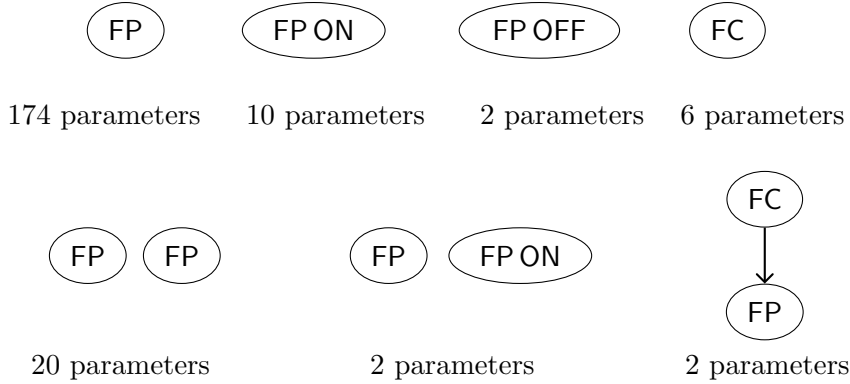
\begin{figure}[h!]
\centering
	\begin{tikzpicture}[main node/.style={ellipse,fill=white!20,draw,scale=1}]
		\node[main node] (FP1) at (-4,0) {$\FP$};
		\node[main node] (FPON1) at (-1.5,0) {$\FPon$};
		\node[main node] (FPOFF1) at (1.5,0) {$\FPoff$};
		\node[main node] (FC1) at (4,0) {$\FC$};
	\end{tikzpicture}

	\vspace{0.2in}

	\begin{tabular}{cccc}
		\hspace{0in}174 parameters & \hspace{.1in}10 parameters & \hspace{.1in} 2 parameters & \hspace{0in} 6 parameters\\
		&&& \\
	\end{tabular}

	\begin{tikzpicture}[main node/.style={ellipse,fill=white!20,draw,scale=1}]
		\node[main node] (FP2) at (-5,-2) {$\FP$};
		\node[main node] (FP3) at (-3.75,-2) {$\FP$};
		\node[main node] (FP4) at (-1,-2) {$\FP$};
		\node[main node] (FPON2) at (0.75,-2) {$\FPon$};
		\node[main node] (FC2) at (4,-1.25) {$\FC$};
		\node[main node] (FP4) at (4,-2.75) {$\FP$};

		\path[->,>=angle 90,thick]
		(FC2) edge[] node[] {} (FP4)
		;
	\end{tikzpicture}
	\begin{tabular}{ccc}
		\hspace{0.2in}20 parameters & \hspace{.7in}2 parameters & \hspace{0.6in} 2 parameters
	\end{tabular}
	\caption{$\mathsf{DSGRN}$ Morse graphs for the bistable repressilator.}
\label{fig:mg_bistable}	
\end{figure}

The collection of Morse graphs immediately signal the presence of richer dynamics across parameter space than for the repressilator system. In addition to the four Morse graphs seen in the repressilator example, there are two new dynamic signatures present: bistability (the existence of two stable fixed points, Type B) and an unstable full cycle with an attracting fixed point (type D).  The comparison of Figures~\ref{fig:mg_repress} and~\ref{fig:mg_bistable} that the addition of a single edge to a regulatory network $\bRN$ can radically change the dynamical signature of $\bRN$ across parameter space.

 From Figure~\ref{fig:mg_bistable}, we see that there are  six parameter vertices with the  Morse graph $\FC$, which suggest a presence of a  stable periodic oscillation. DGSRN database provides us with the inequalities that define these regions in the parameter $Z$. For illustration, we select one of them (parameter 151), which represents a region in $Z \subset \R^{15}$ given  by 
\begin{align}
	l_{1,2}l_{1,3} &<  \left\{ \begin{array}{c} u_{1,2}l_{1,3} \\ l_{1,2}u_{1,3} \end{array} \right \} < \gamma_1 \theta_{2,1} < u_{1,2}u_{1,3} \nonumber \\
	l_{2,1} &< \gamma_2 \theta_{3,2} < u_{2,1} < \gamma_2\theta_{1,2} \label{eq:birepparam}\\
	l_{3,2} &< \gamma_3\theta_{1,3} < u_{3,2}, \nonumber.
\end{align} 
The curly braces denote an undetermined (arbitrary) order:  the relative order of $u_{1,2}l_{1,3}$ and $ l_{1,2}u_{1,3}$ does not change the wall graph and hence the Morse graph as long as both  remain below $\gamma_1 \theta_{2,1}$.

We sample this parameter region  at the values $\gamma_1=\gamma_2=\gamma_3 = 1$, $l_{i,j}=1$, $\theta_{1,3}=2$, $\theta_{3,2}=3$, $\theta_{2,1}=4$, $\theta_{1,2}=6$, $u_{1,2}=2$, $u_{1,3}=3$, $u_{3,2}=4$, and $u_{2,1}=5$. Substituting these values into~\eqref{eq:hillminus}, we have the following system of smooth equations: 
\begin{align} 
 \dot{x}_1 &= -x_1 + \left(1+ \frac{6^n}{6^n + x_2^n}\right)\left(1+ \frac{2^{n+1}}{2^n + x_3^n}\right) \nonumber\\
 \dot{x}_2 &= -x_2 + 1+ \frac{4^{n+1}}{4^n + x_1^n} \label{eq:birephill}\\
 \dot{x}_3 &= -x_3 + 1+ \frac{3^{n+1}}{3^n + x_2^n}. \nonumber
 \end{align}
Results of the  simulations are  shown in Figure~\ref{fig:birephill} for $n=10$.
 There is no condition analogous to the secant condition that would provide  an estimate for $n$ that would  produce a stable periodic orbit. Numerically, we found that at these parameter values $n=7$ is sufficient for periodicity and $n=6$ is not.

\begin{figure}[h!]
\centering
\includegraphics[width=4in]{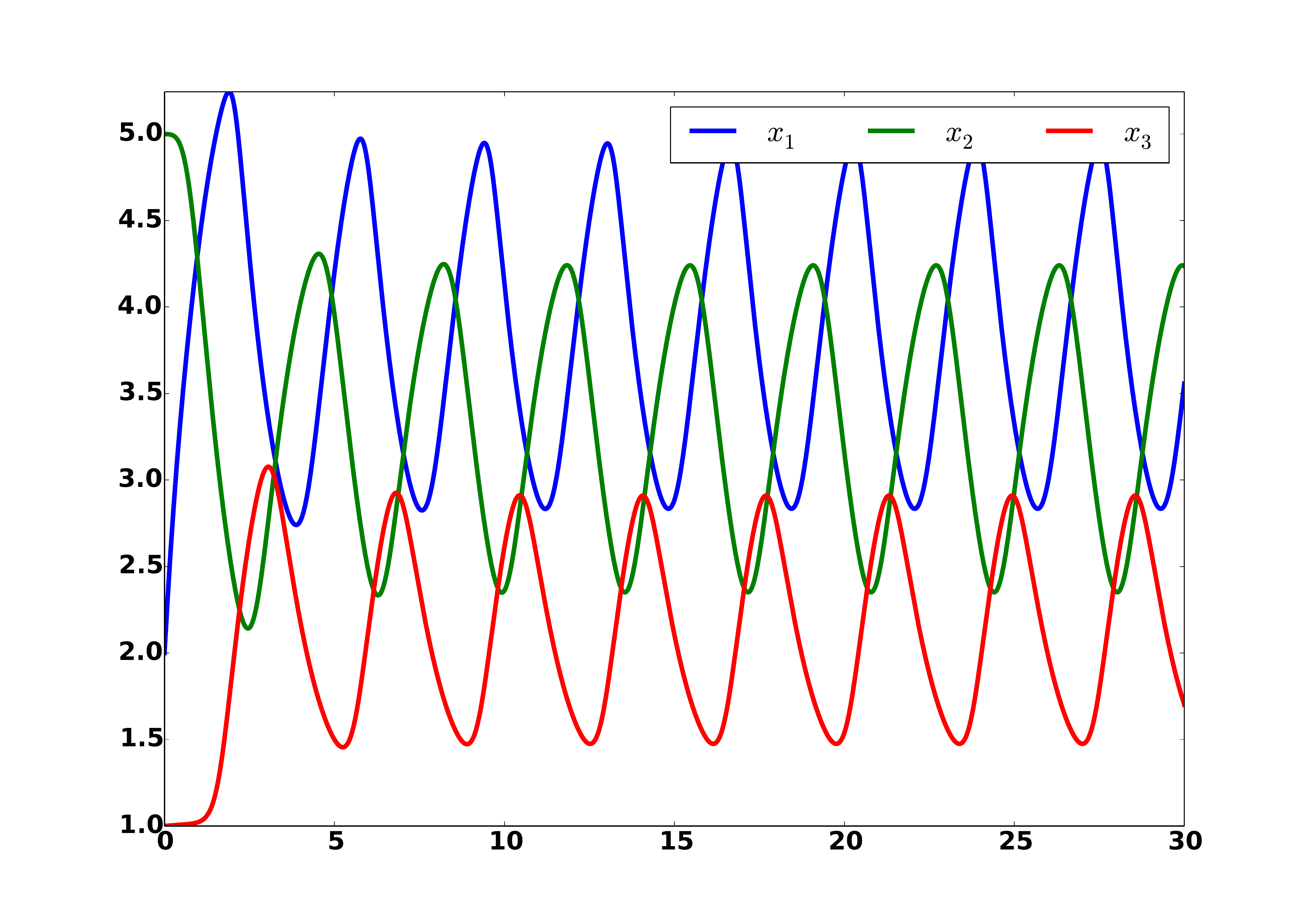}
\caption{Hill function simulation for the bistable repressilator at parameter node 151. See the text for parameter choices. The Hill exponent is $n=10$.}
\label{fig:birephill}
\end{figure}

We want to finish this section with an important observation about a  relationship between the parameter graph for the repressilator and the parameter graph for the bistable repressilator. 
Since the repressilator network is a subnetwork of the bistable repressilator, a natural question is whether there is a similar correspondence between their parameter graphs.
To begin to answer this question  we investigate the parameter node given by (\ref{eq:birepparam}) in bistable repressilator. 
We first note that because $ l_{2,1}  < u_{2,1} < \gamma_2\theta_{1,2}$ the second component of  any  target point in the system  $\frac{\sigma^-_{2,1}(\cdot)}{\gamma_2} < \theta_{1,2}$. 
Therefore, perhaps after a transient,  $x_2 (t) < \theta_{1,2}$ for all $ t\geq T$ and some $T>0$.
Consequently, the value of the function  $\sigma^-_{1,2}(x_2)$ will be  $u_{1,2}$ for all $t \geq T$ and $x_2$ will effectively cease regulation of $x_1$. Therefore the network that is {\em effectively} represented by this parameter node is not bistable repressilator, but a repressilator where the edge from $x_2$ to $x_1$ is erased. 
This brings up a set of interesting questions about how to recognize subnetworks that are effectively represented by each node in the parameter graph, and whether it is possible to build parameter graphs of larger networks from parameter graphs of their subnetworks. 
The answers to these  questions are beyond the scope of this paper but  will  be addressed in the near future. 

\subsection{p53 network}
While the first two examples of this section were aimed at illustrating the main concepts of the paper on small networks, $\mathsf{DSGRN}$ is being used on larger and more complicated networks with greater biological urgency.
We briefly comment on simulations of a subnetwork of the p53 signaling network from~\cite{Loewer2010}. 
The point of including this model in this paper is to indicate the ease with which $\mathsf{DSGRN}$ handles a network of this size and then allows the dynamics to be interrogated.  
 
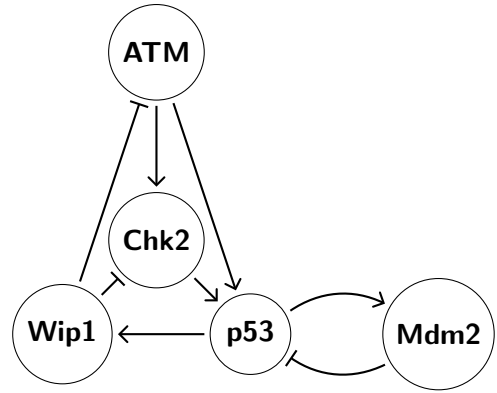
\begin{figure}[h!]
		\begin{center}
			\begin{tikzpicture}[main node/.style={circle,fill=white!20,draw,font=\sffamily\normalsize\bfseries},scale=2.5]
				\node[main node] (Wip1) at (0,0) {Wip1};
				\node[main node] (p53) at (1,0) {p53};
				\node[main node] (Mdm2) at (2,0) {Mdm2};
				\node[main node] (Chk2) at (0.5,0.5){Chk2};
				\node[main node] (ATM) at (0.5,1.5) {ATM};

				\path[->,>=angle 90,thick]
				(Wip1) edge[-|,shorten <= 2pt, shorten >= 2pt] node[] {} (Chk2)
				(Wip1) edge[-|,shorten <= 2pt, shorten >= 2pt] node[] {} (ATM)
				(p53) edge[->,shorten <= 2pt, shorten >= 2pt] node[] {} (Wip1)
				(p53) edge[->,shorten <= 2pt, shorten >= 2pt, bend left] node[] {} (Mdm2)
				(Chk2) edge[->,shorten <= 2pt, shorten >= 2pt] node[] {} (p53)
				(ATM) edge[->,shorten <= 2pt, shorten >= 2pt] node[] {} (p53)
				(ATM) edge[->,shorten <= 2pt, shorten >= 2pt] node[] {} (Chk2)
				(Mdm2) edge[-|,shorten <= 2pt, shorten >= 2pt, bend left] node[]  {} (p53)
				;

			\end{tikzpicture}
			\caption{Subnetwork of key species of the p53 signaling network.}\label{fig:p53}
		\end{center}
\end{figure}

As indicated in Figure~\ref{fig:p53} this network has 5 nodes and 8 edges and thus the parameter space $Z \subset \bar{Z}\subset\R^{29}$. 
As in the previous cases we construct the  geometric parameter graph $\GPG$ by first computing the $\CPG$ using the factors listed in Table~\ref{table:goodnodes}. This results in a $\GPG$ with 803520 nodes. Next we construct a $\mathsf{DSGRN}$ database over this parameter graph. This problem is quite tractable and the database construction took only 37 seconds on a Mid-2014 Macbook Pro laptop (Intel Core i7-4870HQ CPU @ 2.50GHz). 
We note that for larger parameter graphs, our software can scale to HPC cluster environments, but in this case this is clearly not necessary.

We remark that there is interest in the question of stable oscillations in this system \cite{geva-zatorsky:2006}.
Given the size of the $\GPG$, visualization is impractical. 
However, as indicated in the introduction $\mathsf{DSGRN}$ produces an SQL database. 
The query for a Morse graph with a minimal node annotated by $FC$ identifies parameter nodes associated with stable recurrent dynamics where all the species pass thresholds. 
There are 6904 nodes in the $\GPG$ which satisfy this query, of which 3204 are associated with a Morse graph consisting of a single node.
Each of these parameter nodes has a CAD description of the inequalities which defines the associated connected region in parameter space.
Table~\ref{table:cad40535} provides an explicit CAD description for one of these nodes (labeled 40535 in the database \cite{dsgrn}).
from which it easy  to chose  parameter values associated with the parameter node.

\begin{table}[]
	\begin{threeparttable}
		\centering
		\caption{CAD description of a parameter node for the p53 Network.}
		\label{table:cad40535}
		
		\begin{tabular}{| c | c | c | }
			\hline
			Component 1  		& Component 2 			& Component 3 \\
			\hline
			$0 <\text{T}_7$		& $0 <\text{T}_7$		& $0 <\text{T}_7$ \\
			$0 <\text{L}_1$		& $0 <\text{L}_1$		& $0 <\text{L}_1$ \\
			$0 <\text{L}_2$		& $0 <\text{L}_2$		& $0 <\text{L}_2$ \\
			$0 <\text{L}_3< \frac{\text{T}_7}{\text{L}_1+\text{L}_2}$	&	$0 < \text{L}_3  < \frac{\text{T}_7}{\text{L}_1+\text{L}_2}$	&
			$0 < \text{L}_3 < \frac{\text{T}_7}{\text{L}_1+\text{L}_2}$ \\
			$\text{L}_1 < \text{U}_1 \leq \frac{\text{L}_3 ( \text{L}_1 - \text{L}_2 ) + \text{T}_7}{\text{2L}_3}$	& 
			$\frac{\text{L}_3 (\text{L}_1 - \text{L}_2) + \text{T}_7}{2 \text{L}_3} < \text{U}_1 < \frac{\text{T}_7 - \text{L}_2 \text{L}_3}{\text{L}_3}$	&
			$\frac{\text{L}_3 (\text{L}_1 - \text{L}_2) + \text{T}_7}{2 \text{L}_3} < \text{U}_1 < \frac{\text{T}_7-\text{L}_2 \text{L}_3}{\text{L}_3}$ \\
			$\frac{\text{T}_7 - \text{L}_3 \text{U}_1}{\text{L}_3} < \text{U}_2 < \frac{\text{T}_7-\text{L}_1 \text{L}_3}{\text{L}_3}$	& 
			$\frac{\text{T}_7 - \text{L}_3 \text{U}_1}{\text{L}_3} < \text{U}_2 < \text{L}_2 - \text{L}_1 + \text{U}_1$	&	
			$\text{L}_2 + \text{U}_1 - \text{L}_1 \leq \text{U}_2 < \frac{\text{T}_7-\text{L}_1 \text{L}_3}{\text{L}_3}$ \\
			$\text{L}_3 < \text{U}_3 < \frac{\text{T}_7}{\text{L}_1 + \text{U}_2}$	&	$\text{L}_3 < \text{U}_3 < \frac{\text{T}_7}{\text{L}_2+\text{U}_1} $	&
			$\text{L}_3 < \text{U}_3 < \frac{\text{T}_7}{\text{L}_1 + \text{U}_2} $ \\
			$\text{T}_7 < \text{T}_8 < \text{L}_3( \text{U}_1 + \text{U}_2 )$ 	& $\text{T}_7 < \text{T}_8 < \text{L}_3 (\text{U}_1 + \text{U}_2)$	& $\text{T}_7 < \text{T}_8 < \text{L}_3 (\text{U}_1 + \text{U}_2)$ \\
			$0 < \text{T}_2$ 	& $0 < \text{T}_2$	    & $0 < \text{T}_2$ \\
			$0 < \text{L}_5$ 	& $0 < \text{L}_5$		& $0 < \text{L}_5$ \\
			$0 < \text{L}_6 < \frac{\text{T}_2}{\text{L}_5}$	&	$0 < \text{L}_6 < \frac{\text{T}_2}{\text{L}_5}$	&	$0 < \text{L}_6 < \frac{\text{T}_2}{\text{L}_5}$ \\
			$\text{L}_5 < \text{U}_5 < \frac{\text{T}_2}{\text{L}_6}$	&	$\text{L}_5	< \text{U}_5 < \frac{\text{T}_2}{\text{L}_6}$	&
			$\text{L}_5 < \text{U}_5 < \frac{\text{T}_2}{\text{L}_6}$ \\
			$\frac{\text{T}_2}{\text{U}_5} < \text{U}_6 < \frac{\text{T}_2}{\text{L}_5}$	&	$\frac{\text{T}_2}{\text{U}_5} < \text{U}_6 < \frac{\text{T}_2}{\text{L}_5}$	&	$\frac{\text{T}_2}{\text{U}_5} < \text{U}_6 < \frac{\text{T}_2}{\text{L}_5} $	\\
			$0 < \text{T}_6$  	& $0 < \text{T}_6$		& $0 < \text{T}_6$ \\
			$0 < \text{U}_7 < \text{T}_6$	&	$0 < \text{U}_7 < \text{T}_6$	&	$0 < \text{U}_7 < \text{T}_6$ \\
			$0 < \text{L}_7 < \text{U}_7$  	&	$0 < \text{L}_7 < \text{U}_7$	&	$0 < \text{L}_7 < \text{U}_7$ \\
			$\text{L}_7 < \text{T}_4 < \text{U}_7$		& $\text{L}_7 < \text{T}_4 < \text{U}_7$	&	$\text{L}_7 < \text{T}_4 < \text{U}_7$ \\
			$0 < \text{L}_4$ 	& $0 < \text{L}_4$		& $0 < \text{L}_4$	\\ 
			$\text{L}_4 < \text{U}_4$		&	$\text{L}_4 < \text{U}_4$	  	&	$\text{L}_4 < \text{U}_4$ \\
			$\text{L}_4 < \text{T}_1 < \text{U}_4$		& $\text{L}_4 < \text{T}_1 < \text{U}_4$	&	$\text{L}_4 < \text{T}_1 < \text{U}_4$ \\
			$0 < \text{T}_3$ 	& $0 < \text{T}_3$		& $ 0 < \text{T}_3$	\\
			$\text{T}_1 < \text{T}_5 < \text{U}_4$		& $\text{T}_1 < \text{T}_5 < \text{U}_4$	&	$\text{T}_1 < \text{T}_5 < \text{U}_4$ \\
			$0 < \text{L}_8 < \text{T}_3$ 	& $0 < \text{L}_8 < \text{T}_3$		&	$ 0 < \text{L}_8 < \text{T}_3$ \\
			$\text{T}_3 < \text{U}_8$	 	& $\text{T}_3 < \text{U}_8$			&	$\text{T}_3 < \text{U}_8$ \\
			\hline
		\end{tabular}	
		\begin{tablenotes}
			\item We use the following numeric scheme to identify the edges: \\
			1 = (ATM $\to$ p53), 2 = (Chk2 $\to$ p53), 3 = (Mdm2 $\to$ p53), 4 = (Wip1 $\to$ ATM), 
			5 = (ATM $\to$ Chk2), 6 = (Wip1 $\to$ Chk2), 7 = (p53 $\to$ Wip1), 8 = (p53 $\to$ Mdm2).\\
			The upper values, lower values, and the thresholds for each edge correspond to 
			$\text{U}_i$, $\text{L}_i,$ and $\text{T}_i$ respectively. \\
			Example: The upper value associated with the edge (ATM $\to$ Chk2) corresponds to $\text{U}_5$.
		\end{tablenotes}
		
	\end{threeparttable}
\end{table}

\begin{table}[]
\centering
\caption{P53 Network parameters}
\label{table:parameters}
\begin{tabular}{|l|c|c|c|}
\hline
  Edge  &  u -value & l-value & threshold \\
 \hline
 ATM $\to$ Chk2 & $1$ & $1/2$  & $1/2$  \\
 ATM $\to$ p53   &  $7/8$   & $7/32$ &  $1/4$ \\
 Chk2 $\to$ p53   &    $7/8$   & $7/32$ &  $3/4$ \\
 MdM2$\to$ p53   &   $7/8$    & $21/32$ &  $1$ \\
 Wip1 $\to$ ATM  &    $1$   & $1/2$ &  $1/2$ \\
 Wip1 $\to$ Chk2  &  $2$   & $1/2$ &   $2$\\
 p53 $\to$ Mdm2  &  $2$   & $1/2$ &   $1127/1024$ \\
 p53 $\to$ Wilp1  &  $1$   & $1/4$ &   $ 539/512$\\
  \hline
\end{tabular}
\end{table}

As  in the previous examples we perform a numerical simulation of this system using Hill functions.
Using the CAD description of Table~\ref{table:cad40535} we choose  decay rates $\gamma_i = 1$ and the remaining parameters of the switching system as presented in Table~\ref{table:parameters}.
It remains to choose the Hill exponent $n$ for each nonlinearity i.e.\ one exponent for each edge in the network. 
Setting all Hill exponents  to be $8$ the solution exhibits the oscillations depicted in Figure~\ref{fig:p53Hill}. 
The uniform choice of $n=6$ does not produce oscillations, but many other choices, e.g.\ setting $n=2$ for the connections Mdm2 $\to$ p53 and p53 $\to$ Mdm2  and $n=10$ for all other nonlinearities also produces oscillations. 
It is worth noting that the peaks of p53  in Figure~\ref{fig:p53Hill} come slightly ahead of the peaks of Mdm2 which agrees with one of the key  experimental observations in ~\cite{Loewer2010}.

\begin{figure}[h!]
\centering
\includegraphics[width=4in]{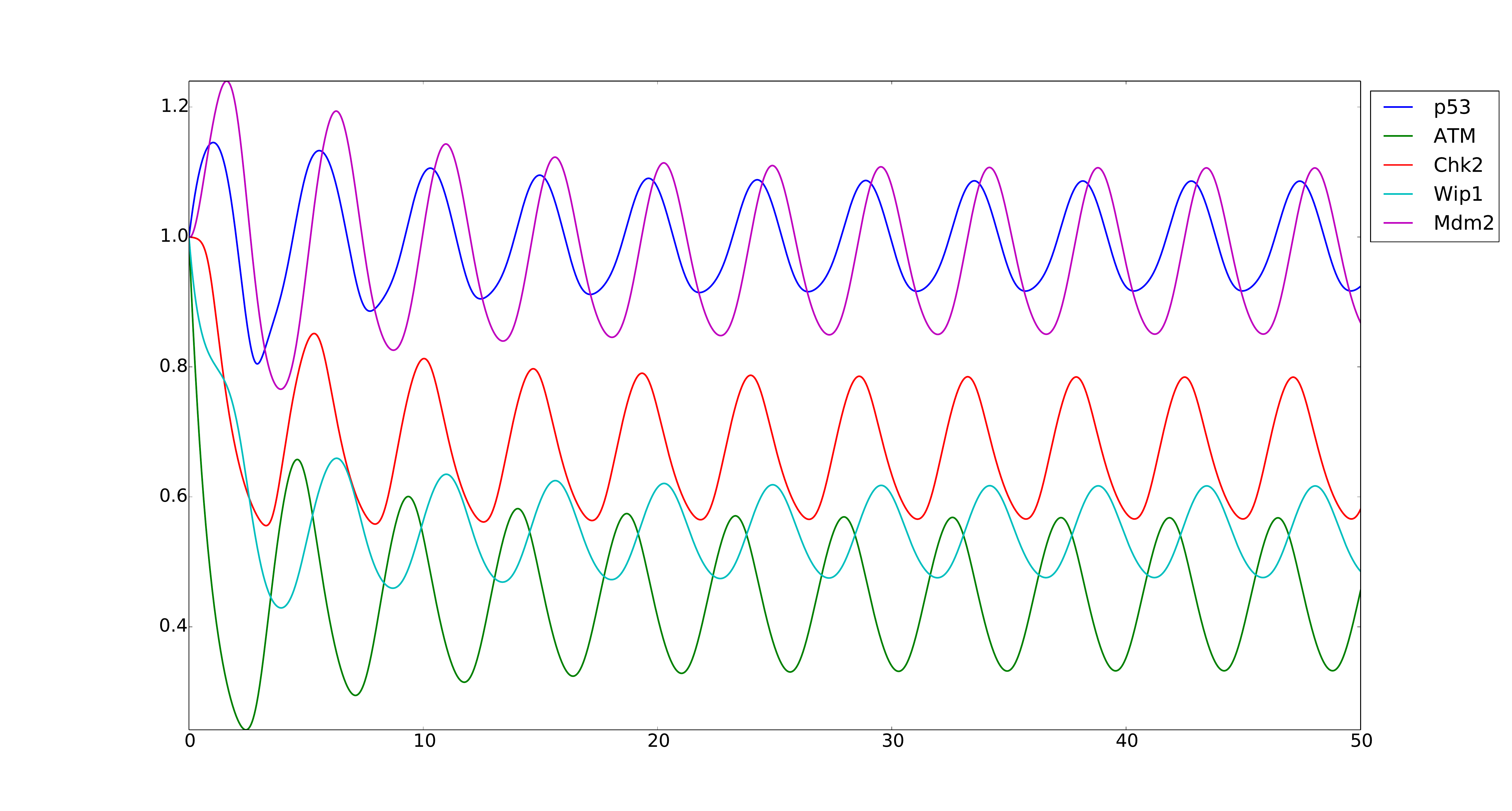}
\caption{Hill function simulation for the p53 model at parameter node 40535. See the text for parameter choices. The Hill exponent for every nonlinearity is $n=8$.}
\label{fig:p53Hill}
\end{figure}

This last example is meant to indicate the usefulness of the CAD description of the nodes in parameter space.
It leads to trivial algorithms for choosing specific parameter points.
Therefore, this opens the possibility for efficiently studying numerical simulations to understand finer structure of the dynamics that satisfies the local and global properties of the annotated Morse graph. 

\bigskip
\noindent{\bf Acknowledgements:}
We thank Chang Chan for bringing the p53 network to our attention.
T.G. was partially supported by  NSF grants DMS-1226213, DMS-1361240, DARPA D12AP200025 and NIH R01 grant 1R01AG040020-01.
B.C. was partially supported by  DARPA  grant D12AP200025.
K.M. and S.H. were partially supported by NSF-DMS-0835621, 0915019, 1125174, 1248071, and contracts from AFOSR and DARPA.

%

\bibliographystyle{plain}
\bibliography{switching_bib}

\end{document}